\documentclass{amsart}
\usepackage{amsmath,amssymb}

\newtheorem{theorem}{Theorem}[section]

\newtheorem{corollary}[theorem]{Corollary}
\theoremstyle{definition}

\theoremstyle{remark}
\newtheorem{remark}[theorem]{Remark}

\numberwithin{equation}{section}

\begin{document}

\title[Optimal Monotonicity of $L^p$ Integral of Green Function]{Optimal Monotonicity of $L^p$ Integral of Conformal Invariant Green Function}

\author{Jie Xiao}
\address{Department of Mathematics and Statistics, Memorial University of Newfoundland, St. John's, NL A1C 5S7, Canada}
\email{jxiao@mun.ca}
\thanks{The author was supported in part by Natural Science and Engineering Research Council of Canada.}



\date{}



\begin{abstract}
Both analytic and geometric forms of an optimal monotone principle
for $L^p$-integral of the Green function of a simply-connected
planar domain $\Omega$ with rectifiable simple curve as boundary are
established through a sharp one-dimensional power integral estimate
of Riemann-Stieltjes type and the Huber analytic and geometric
isoperimetric inequalities under finiteness of the positive part of
total Gauss curvature of a conformal metric on $\Omega$.
Consequently, new analytic and geometric isoperimetric-type
inequalities are discovered. Furthermore, when applying the
geometric principle to two-dimensional Riemannian manifolds, we find
fortunately that $\{0,1\}$-form of the induced principle is midway
between Moser-Trudinger's inequality and Nash-Sobolev's inequality
on complete noncompact boundary-free surfaces, and yet equivalent to
Nash-Sobolev's/Faber-Krahn's
eigenvalue/Heat-kernel-upper-bound/Log-Sobolev's inequality on the
surfaces with finite total Gauss curvature and quadratic area
growth.
\end{abstract}
\maketitle

\section{Introduction}

Given a conformal metric of the form
$$
\mathsf{\sigma}=e^{2u}ds^2=e^{2u}|dz|^2=e^{2u}(dx^2+dy^2)
$$
for $z=x+iy$ in a subdomain $\Sigma$ of the two dimensional
Euclidean space $\mathbb R^2$, we are mainly inspired by Huber's
1957 Acta Math. paper ``Zur isoperimetrischen Ungleichung auf
gekrümmten Flächen" \cite{Hu57} and 1954 Ann. Math. paper ``On the
isoperimetric inequality on surfaces of variable Gaussian curvature"
\cite{Hu54} to establish a sharp monotone principle for the power
$p\in [0,\infty)$ integral (as well as its limiting case
$p\to\infty$)
$$
\big(\Gamma(1+p)\big)^{-1}\Big(4\pi\big(1-(2\pi)^{-1}{\int_\Omega\max\{K_\mathsf{\sigma},0\}dA_\mathsf{\sigma}}\big)\Big)^p\int_{\Omega}\big(g_{(\Omega,\sigma)}(\cdot,a)\big)^p\,
dA_\mathsf{\sigma}(\cdot),\quad a\in \Omega
$$
of the conformally invariant Green function
$g_{(\Omega,\sigma)}(\cdot,\cdot)$ for the two-dimensional conformal
Laplacian
$$
\Delta_\mathsf{\sigma} u=e^{-2u}\Delta
u
$$
of a simply-connected domain $(\Omega,\mathsf{\sigma})$ on the
surface $(\Sigma, \mathsf{\sigma})$ with a rectifiable simple curve
as its boundary -- see Theorem \ref{t4b}. Here and henceforth
$$
K_\mathsf{\sigma}(z)=-e^{-2u(z)}\Delta
u(z)=-e^{-2u(z)}\left(\frac{\partial^2 u(z)}{\partial
x^2}+\frac{\partial^2 u(z)}{\partial y^2}\right)
$$
and
$$
dA_\mathsf{\sigma}(z)=e^{2u(z)}dA(z)=e^{2u(z)}dxdy
$$
are the Gauss curvature and the area element of the surface
$(\Sigma,\mathsf{\sigma})$ respectively. Of course, $\Gamma(\cdot)$
is the classical gamma function.

To reach this geometric principle we will first consider its
equivalent analytic form -- Theorem \ref{t4}. This extends sharply
the following result of Stanton \cite{St}:

\begin{theorem}\label{t1} Let $\Phi$ be of class $C^2$ with $\Delta\Phi>0$ on a simply-connected domain $\Omega\subset\mathbb R^2$ with $\partial\Omega$ being a rectifiable simple curve. If
$$
\int_\Omega\max\Big\{\frac{\Delta
\ln\big(\Delta\Phi(z)\big)^{-1}}{\Delta\Phi(z)},0\Big\}dA(z)<2\pi,
$$
then for $a\in\Omega$,
\begin{equation}\label{eq1}
\int_\Omega g_\Omega(z,a)\Delta\Phi(z) dA(z)\le
\frac{\int_\Omega
\Delta\Phi(z) dA(z)}{4\pi\big(1-(2\pi)^{-1}{\int_\Omega\max\big\{\frac{\Delta
\ln(\Delta\Phi(z))^{-1}}{\Delta\Phi(z)},0\big\}dA(z)}\big)}.
\end{equation}
\end{theorem}

\begin{remark}\label{r1} In the case of $\Delta\Phi=1$ the inequality (\ref{eq1}) is back to the
so-called P\'olya-Szeg\"o's ``stress" inequality -- see also
\cite[p. 115, (12)]{PoSz}:
\begin{equation}\label{eq2}
\int_\Omega g_\Omega(z,a)dA(z)\le \frac1{4\pi}\int_{\Omega}dA(z),
\end{equation}
which was generalized by Bandle \cite[p. 61, Example 1]{Ba} to the
inequality

\begin{equation}\label{eq3}
\int_\Omega \big(g_\Omega(z,a)\big)^pdA(z)\le
\frac{\Gamma(p+1)}{(4\pi)^p}\int_{\Omega}dA(z),\quad p\in [0,\infty).
\end{equation}
The constants in (\ref{eq3}) and (\ref{eq2}) are sharp since they
are attained when $\Omega$ is any Euclidean disk centered at $a$.
Interestingly, (\ref{eq3}) becomes a special case of Aulaskar-Chen's
``$Q_p$-norm" inequality (cf. \cite{AuCh}):

\begin{equation}\label{eq4}
\int_\Omega \big(g_\Omega(z,a)\big)^{p}|f'(z)|^2dA(z)\le
\frac{(4\pi)^{-p}\Gamma(p+1)}{(4\pi)^{-q}\Gamma(q+1)}\int_{\Omega}\big(g_\Omega(z,a)\big)^{q}|f'(z)|^2dA(z)
\end{equation}
which is valid for all $0\le q<p<\infty$, $a\in\Omega$, and
holomorphic functions $f$ on $\Omega$. It is also worth remarking
that the equality in (\ref{eq4}) holds under convergence of the
right-hand integral of (\ref{eq4}) if and only if $\Omega$ is a
simply-connected domain $\Lambda\subset\mathbb R^2$ minus at most a
compact set $E$ of logarithmic capacity zero --
$$
0=\hbox{cap}(E)=\exp\left(-\inf_\mu\int_{\mathbb{R}^2}\int_{\mathbb{R}^2}\Big(\ln\frac{1}{|z-w|}\Big)d\mu(z)d\mu(w)\right)
$$
(cf. \cite[p. 25]{SaT}) where the infimum ranges over all positive
probability Radon measures $\mu$ supported on $E$, but also $f$ can
be extended to a conformal mapping from $\Lambda$ onto an open disk
in $\mathbb{R}^2$ centered at $f(a)$.
\end{remark}

In order to prove the equivalent principle we introduce a process
that reduces the desired optimal estimate to a one-dimensional
calculus inequality in connection with the so-called
Riemann-Stieltjes integral -- see Theorem \ref{l6}.

Finally, we apply our ideas, methods and techniques to explore an
analogue of the geometric monotone principle on two dimensional
simply-connected, complete, noncompact and boundary-free Riemannian
manifolds with $2\pi$-bounded total Gauss curvature -- Theorem
\ref{t51}, thereby surprisingly finding that with generic constants,
the $0=p_1<p_2=1$ setting of Theorem \ref{t51} lies nicely between
Moser-Trudinger's inequality (cf. Adams's 1988 Ann. Math. paper ``A
sharp inequality of J. Moser for higher order derivatives"
\cite{Ads} for an account) and Nash-Sobolev's inequality (cf.
Chavel's 2001 and Saloff-Coste's 2002 Cambridge Univ. Press books
``Isoperimetric Inequalities" \cite{Ch} and ``Aspects of
Sobolev-Type Inequalities" \cite{Sal} for instance) on complete
noncompact surfaces without boundary; but this special case is also
equivalent to the generic Nash-Sobolev's/Faber-Krahn's eigenvalue
inequality/Heat-kernel-upper-bound inequality/Log-Sobolev's
inequality on the surfaces with finite total Gauss curvature and
quadratic area growth (cf. Li-Tam's 1991 J. Diff. Geom. paper
``Complete surfaces with finite total curvature" \cite{LiTa} for
more information on such a kind of surfaces) -- Theorem \ref{t52}.

We would like to take this opportunity to thank G. Zhang for his
suggestion on the first version of the paper, encouraging us to
explore a useful application of the original principle. At the same
time, we are grateful to P. Li and K. Zhu for sending us their nice
articles \cite{LiTa} and \cite{Zh} as two important references. Last
but not least, it is our pleasure to acknowledge some related
discussions with A. Chang and J. Qing during 2008 Univ. Arkansas
conference on ``Partial Differential Equations in Conformal
Geometry".

\section{Optimal Monotonicity -- Basic Form}

In this section we establish a sharp one-dimensional inequality for
the Riemann-Stieltjes $L^p$ integral of the radial function -- that
is -- Theorem \ref{l6} below. This useful and fundamental result
seems to be of independent interest although some basic techniques
used to argue its special case $c=2$ have a root in Aulaskar-Chen's
\cite[Lemma 2]{AuCh}. Actually, it is a key step to the principles
which will be precisely presented in the subsequent sections.

\begin{theorem}\label{l6} Given a constant $c>0$ and a nonnegative function $X(\cdot)$ on $(0,\infty)$,
suppose
 \begin{equation}\label{eql}
 X'(t)=\frac{dX(t)}{dt}\le 0\quad\hbox{and}\quad
 \frac{d\big(e^{ct}X(t)\big)}{dt}\le 0\quad\hbox{for}\quad t>0.
 \end{equation}
For $p\in [0,\infty)$ let $Y_p(t)=-\int_t^\infty r^p dX(r)$ be
defined on $[0,\infty)$ in the sense of Riemann-Stieltjes
integration.

\item{\rm(i)} If $0\le p_1<p_2<\infty$, then
\begin{equation}\label{eq8}
\frac{c^{p_2}Y_{p_2}(0)}{\Gamma(p_2+1)}\le\frac{c^{p_1}Y_{p_1}(0)}{\Gamma(p_1+1)}.
\end{equation}
Here
\begin{equation}\label{eq9}
\frac{c^{p_2}Y_{p_2}(0)}{\Gamma(p_2+1)}=\frac{c^{p_1}Y_{p_1}(0)}{\Gamma(p_1+1)}<\infty
\end{equation}
if and only if
$$
X(0)=\lim_{r\to 0^+}X(r)<\infty\quad\hbox{and}\quad
X(t)=e^{-ct}X(0)\ \ \hbox{for}\ \ t\ge 0.
$$

\item{\rm(ii)} If $Y_{p_0}(0)<\infty$ is valid for some $p_0\in
[0,\infty)$, then
\begin{equation}\label{eq10a}
\lim_{t\to\infty}e^{ct}X(t)=\lim_{p\to\infty}\frac{c^pY_p(0)}{\Gamma(p+1)}.
\end{equation}
\end{theorem}

\begin{proof} (i) The supposition $X'(t)\le 0$ (where $t>0$) makes both $Y_{p_1}(0)$ and $Y_{p_2}(0)$ meaningful.
Without loss of generality we may assume $Y_{p_1}(0)<\infty$ since
$Y_{p_1}(0)=\infty$ implies that (\ref{eq8}) is trivially true. If
$p_1=0$ then $Y_{p_1}(t)=X(t)$ follows from $d(e^{ct}X(t))/dt\le 0$.
Consequently,
$$
\frac{d Y_0(t)}{Y_0(t)}\le -cdt=-\frac{t^0e^{-ct}dt}{\int_t^\infty
r^0e^{-cr}dr},\quad t>0.
$$
If $p_1>0$, then both $d(e^{ct}X(t))/dt\le 0$ and
integration-by-part imply that for $t>0$,
\begin{eqnarray*}
Y_{p_1}(t)&=& t^{p_1}X(t)+p_1\int_t^\infty r^{p_1-1}X(r)dr\\
&\le& X(t)\left(t^{p_1}+p_1e^{ct}\int_t^\infty r^{p_1-1}e^{-cr}dr\right)\\
&=&cX(t)e^{ct}\int_t^\infty r^{p_1}e^{-cr}dr.
\end{eqnarray*}
As a result, we read off:
$$
\frac{dY_{p_1}(t)}{Y_{p_1}(t)}\le-\frac{ct^{p_1}X(t)dt}{Y_{p_1}(t)}\le-\frac{t^{p_1}e^{-ct}dt}{\int_t^\infty
r^{p_1} e^{-cr}dr},\quad t>0.
$$
Integrating this inequality from $0$ to $t$, we obtain
$$
Y_{p_1}(t)\le\frac{c^{p_1+1}Y_{p_1}(0)}{\Gamma(p_1+1)}\int_t^\infty
r^{p_1} e^{-cr}dr,\quad t\ge 0.
$$
With the help of the above estimates we have that for $0\le
p_1<p_2<\infty$,
\begin{eqnarray}\label{eqn21}
Y_{p_2}(0)&=&(p_2-p_1)\int_0^\infty t^{p_2-p_1-1}Y_{p_1}(t)dt\nonumber\\
&\le&\frac{c^{p_1+1}(p_2-p_1)Y_{p_1}(0)}{\Gamma(p_1+1)}\int_0^\infty t^{p_2-p_1-1}\Big(\int_t^\infty r^{p_1}e^{-cr}dr\Big)dt\\
&=&c^{p_1-p_2}\frac{\Gamma(p_2+1)}{\Gamma(p_1+1)}Y_{p_1}(0)\nonumber,
\end{eqnarray}
thereby getting (\ref{eq8}).

Regarding the second conclusion of (i), we consider two aspects. On
the one hand, if
$$
X(0)=\lim_{t\to 0^+}X(t)<\infty\quad\hbox{and}\quad
X(t)=X(0)e^{-ct}\quad\hbox{for}\quad t>0,
$$
then
$$
Y_{p}(0)=c^{-p}\Gamma(p+1)X(0)<\infty\quad\hbox{for}\quad\hbox{any}\quad
p\in [0,\infty),
$$
and accordingly (\ref{eq9}) holds. On the other hand, assume
(\ref{eq9}) is valid. From the above treatment it follows that
$Y_{p_1}(0)<\infty$ ensures $X(0)=\lim_{t\to 0^+}X(t)<\infty$. If
the statement ``$X(t)=e^{-ct}X(0)$ for $t\ge 0$" is not true, there
there are two positive numbers $r_0$ and $t_0$ such that $r_0>t_0$
and $X(r_0)<e^{-c(r_0-t_0)}X(t_0)$ hold, and hence the continuity of
$X(\cdot)$ produces such a constant $\delta>0$ that
$X(r_0)<e^{-c(r_0-t)}X(t)$ whenever $t\in (t_0-\delta,t_0]$.
Therefore $d\big(e^{ct}X(t)\big)/dt\le 0$ is applied to derive that
$X(r)<e^{-c(r-t)}X(t)$ as $t\in (t_0-\delta,t_0]$ and $r\ge r_0$.
Consequently, we obtain
$$
Y_{p_1}(t)<cX(t)e^{ct}\int_t^\infty
r^{p_1}e^{-cr}dr\quad\hbox{when}\quad t\in (t_0-\delta,t_0],
$$
whence finding
$$
Y_{p_1}(t)<\frac{c^{p_1+1}Y_{p_1}(0)}{\Gamma(p_1+1)}\int_t^\infty
r^{p_1} e^{-cr}dr\quad\hbox{for}\quad t\in (t_0-\delta,t_0].
$$
This, along with (\ref{eqn21}), yields
$$
Y_{p_2}(0)=(p_2-p_1)\int_0^\infty t^{p_2-p_1-1}Y_{p_1}(t)dt<c^{p_1-p_2}\frac{\Gamma(p_2+1)}{\Gamma(p_1+1)}Y_{p_1}(0)<\infty,
$$
contradicting the previous equality assumption.

(ii) Suppose $Y_{p_0}(0)<\infty$ holds for some $p_0\in [0,\infty)$.
From the argument for (i) we see that $Y_{p}(0)<\infty$ is valid for
all $p\ge p_0$ and so that via integration-by-parts and
$d\big(e^{ct}X(t)\big)/dt\le 0$,
\begin{eqnarray*}
Y_{p}(0)&=&p\int_0^\infty r^{p-1}X(r)dr\\
&=&p\int_0^\infty e^{cr}X(r)r^{p-1}e^{-cr}dr\\
&=&\left. p\Big(e^{ct}X(t)\int_0^t
r^{p-1}e^{-cr}\,dr\Big)\right|_0^\infty
-p\int_0^\infty\Big(\int_0^t r^{p-1}e^{-cr}\,dr\Big)\,d(e^{ct}X(t))\\
&=&\frac{\Gamma(p+1)}{c^p}\lim_{t\to\infty}e^{ct}X(t)-p\int_0^\infty\Big(\int_0^tr^{p-1}e^{-cr}\,dr\Big)\,d(e^{ct}X(t)).
\end{eqnarray*}
Therefore, the desired limit formula (\ref{eq10a}) follows from
verifying that
$$
0\ge I(p,c)=\frac{c^pp}{\Gamma(p+1)}\int_0^\infty\Big(\int_0^t
r^{p-1}e^{-cr} \,dr\Big)\,d(e^{ct}X(t))\to 0\quad\hbox{as}\ \
p\to\infty.
$$
Notice that the condition $d\big(e^{ct}X(t)\big)/dt\le 0$ deduces
that for any $\epsilon>0$ there exists a $t_0>0$ such that
$-\epsilon<\int_{t_0}^\infty d(e^{ct}X(t))\le 0$. So
$$
I_1(p,c)=\frac{c^pp}{\Gamma(p+1)}\int_{t_0}^\infty\Big(\int_0^t
r^{p-1}e^{-cr} \,dr\Big)\,d\big(e^{ct}X(t)\big)\ge\int_{t_0}^\infty
d\big(e^{ct}X(t)\big)>-\epsilon.
$$
Meanwhile, integrating by parts derives
\begin{eqnarray*}
I_2(p,c)&=&\frac{c^pp}{\Gamma(p+1)}\int_0^{t_0}\Big(\int_0^t
r^{p-1}e^{-cr}
\,dr\Big)\,d\big(e^{ct}X(t)\big)\\
&\ge&\frac{c^p}{\Gamma(p+1)}\int_0^{t_0}t^p\,d\big(e^{ct}X(t)\big)\\
&\ge&\frac{c^p}{\Gamma(p+1)}\int_0^{t_0}t^pe^{ct}\,dX(t)\\
&\ge&\frac{c^pe^{ct_0}t_0^{p-p_0}}{\Gamma(p+1)}\int_0^{t_0}t^{p_0}\,dX(t)\\
&\ge&-\frac{c^pe^{ct_0}t_0^{p-p_0}Y_{p_0}(0)}{\Gamma(p+1)}\\
&\to& 0\quad\hbox{as}\ \ p\to\infty.
\end{eqnarray*}
The estimates on $I_1(p,c)$ and $I_2(p,c)$, along with
$d\big(e^{ct}X(t)\big)/dt\le 0$, imply that
$$
-2\epsilon<I(p,c)=I_1(p,c)+I_2(p,c)\le 0
$$
holds for sufficiently large $p$. Thus, $\lim_{p\to\infty}I(p,c)=0$,
as required.
\end{proof}

\begin{remark}\label{r22} A close look at
(\ref{eq8})-(\ref{eq9})-(\ref{eq10a}) leads us to conjecture that
Theorem \ref{l6} (i) is still valid for $-1<p_1<0$. This thought is
also supported by the following analysis:

\item{\rm(i)} Using the Cauchy-Schwarz inequality and (\ref{eq8}), we find that for $-1<p_1<0$,
$$
Y_0(0)\le\big(Y_{p_1}(0)\big)^\frac12\big(Y_{-p_1}(0)\big)^\frac12\\
\le\big(Y_{p_1}(0)\big)^\frac12\big(c^{p_1}\Gamma(1-p_1)Y(0)\big)^\frac12,
$$
thereby getting
\begin{equation}\label{eq10b}
Y_0(0)\le c^{p_1}\Gamma(1-p_1)Y_{p_1}(0)=\Big(\frac{\pi p_1}{\sin\pi
p_1}\Big)\Big(\frac{c^{p_1}}{\Gamma(1+p_1)}\Big)Y_{p_1}(0).
\end{equation}
The estimate (\ref{eq10b}) and the H\"older inequality yield that
for $-1<p_1<p_2<0$,
\begin{equation}\label{eq10c}
Y_{p_2}(0)\le\min\left\{\Big(\frac{\pi p_1 c^{p_1}}{(\sin\pi
p_1)\Gamma(1+p_1)}\Big)^{1-\frac{p_2}{p_1}},\Big(\frac{\pi p_2
c^{p_2}}{(\sin\pi
p_2)\Gamma(1+p_2)}\Big)^{\frac{p_2}{p_1}-1}\right\}Y_{p_1}(0).
\end{equation}
However, when $X(t)=e^{-ct}X(0)$, the equalities in (\ref{eq10b})
and (\ref{eq10c}) do not occur.

\item{\rm(ii)} Noticing that
$$
\lim_{p\to -1^+}\frac{\Gamma(1+p)}{(1+p)^{-1}}=1;\quad 0\le
-\int_1^\infty \frac{dX(r)}{r^{-p}}<\infty\ \ \hbox{for}\ \ p\in
(-1,0);\quad cX(t)\le -X'(t),
$$
and that $(1+p)r^{p} dr$, as a measure on $[0,1]$, converges weakly
to the point mass at $r=0$ as $p\to -1^+$, we achieve
\begin{eqnarray}
\lim_{p\to-1^+}\frac{c^{p}
Y_{p}(0)}{\Gamma(1+p)}&=&-\lim_{p\to-1^+}\frac{c^{p}}{\Gamma(1+p)}\left(\int_0^1
r^{p} dX(r)+\int_1^\infty r^{p} dX(r)\right)\nonumber\\
&=&-c^{-1}\lim_{p\to -1^+}\frac{\int_0^1
X'(r)dr^{1+p}}{(1+p)\Gamma(1+p)}\\
&=&-c^{-1}\lim_{t\to 0^+}X'(t)\ge X(0)\nonumber.
\end{eqnarray}
\end{remark}

\section{Optimal Monotonicity -- Analytic Form}

We first recall a definition of the well-known Green function of a
bounded domain and its corresponding Robin function. Suppose
$\Omega$ is a bounded domain of $\mathbb R^2$ with boundary
$\partial\Omega$. Given $a\in\Omega$, the Green function
$g_\Omega(\cdot,a)$ of $\Omega$ is the solution of the following
Dirichlet boundary problem:
\[
\left\{\begin{array} {r@{\quad,\quad}l}
\Delta g_\Omega(z,a)=-\delta_a(z) & z\in\Omega,\\
g_\Omega(z,a)=0 & z\in \partial\Omega.
\end{array}
\right.
\]
Here $\delta_a(\cdot)$ is the Dirac measure at $a\in\Omega$. Such a
solution may be evaluated by
$$
g_\Omega(z,a)=-(2\pi)^{-1}\Big(H_\Omega(z,a)+\ln{|z-a|}\Big),
$$
where $H_\Omega(\cdot,a)$ is a harmonic function (i.e., $\Delta
H(\cdot,a)=0$) with the same values as $-\ln|\cdot-a|$ on
$\partial\Omega$ -- this gives the Robin's function/mass
$H_\Omega(a,a)$ and the conformal radius $R_\Omega(a)$ of $\Omega$
at $a\in\Omega$:
$$
H_\Omega(a,a)=-2\pi\lim_{z\to a}\Big((2\pi)^{-1}\ln
{|z-a|}+g_\Omega(z,a)\Big)
$$
and
$$
R_\Omega(a)=\exp\big(-H_\Omega(a,a)\big).
$$

In virtue of the fact that if $u$ is of class $C^1$ on $\Omega$ and
its second-order partial derivatives are piecewise continuous on
$\Omega$ and if $u$ is continuous on $\Omega\cup\partial\Omega$ then
$$
u(a)=u_0(a)-\int_\Omega g_\Omega(z,a)\Delta u(z)dA(z),\quad
a\in\Omega
$$
where $u_0$ is the solution of the above Dirichlet problem for
$\Omega$ with the same values $u$ on $\partial\Omega$, Huber proved
the following assertion -- \cite[Theorem 2]{Hu54}:

\begin{theorem}\label{t3} Let $\Omega$ be the interior of a rectifiable simple curve $\partial\Omega$ in $\mathbb R^2$. Suppose $u$ is continuous on $\Omega\cup\partial\Omega$ and of class $C^1$
as well as its second-order derivatives are piecewise continuous on
$\Omega$. Then
\begin{equation}\label{eq5}
\left(\int_{\partial\Omega}e^{u(z)}dL(z)\right)^2\ge
4\pi\Big(1-{(2\pi)^{-1}\int_\Omega\max\{-\Delta
u(z),0\}dA(z)}\Big)\int_\Omega e^{2u(z)}dA(z),
\end{equation}
with equality when and only when $u$ is $\ln|f'|$ of a conformal map
$f$ from $\Omega$ onto a Euclidean disk in $\mathbb R^2$.
\end{theorem}

Here and later on, $dL(z)$ stands for the arc-length element. Below
is our optimal analytic principle for monotonicity of the $L^p$
integral of Green's function with respect to the conformal area
measure $e^{2u}dA$.

\begin{theorem}\label{t4} Let $\Omega$ be the interior of a rectifiable simple curve $\partial\Omega$ in $\mathbb R^2$. Suppose $u$ is continuous on $\Omega\cup\partial\Omega$ and of class $C^1$
as well as its second-order derivatives are piecewise continuous on
$\Omega$. Set
$$
p\in [0,\infty),\quad a\in\Omega,\quad
\kappa(\Omega)=1-(2\pi)^{-1}\int_\Omega\max\{-\Delta
u(z),0\}dA(z)>0,
$$
and
$$
\mathcal{F}\big(p,a,\kappa(\Omega)\big)=\frac{\big(4\pi\kappa(\Omega)\big)^p}{\Gamma(p+1)}\int_{\Omega}\big(g_\Omega(z,a)\big)^pe^{2u(z)}
dA(z).
$$
Then

\item{\rm(i)}
\begin{equation}\label{eq6}
0\le
p_1<p_2<\infty\Rightarrow\mathcal{F}\big(p_2,a,\kappa(\Omega)\big)\le\mathcal{F}\big(p_1,a,\kappa(\Omega)\big).
\end{equation}
The second equality in (\ref{eq6}) occurs when
and only when there exists a conformal map $f$ from $\Omega$ onto a
Euclidean disk centered at $f(a)$ in $\mathbb R^2$ such that
$u=\ln|f'|$.

\item{\rm(ii)}
\begin{equation}\label{eq+}
\lim_{p\to\infty}\mathcal{F}\big(p,a,\kappa(\Omega)\big)=\left\{\begin{array}
{r@{\;,\quad}l}
0 & \kappa(\Omega)<1,\\
\pi\big(e^{u(a)}R_\Omega(a)\big)^2 & \kappa(\Omega)=1,
\end{array}
\right.
\end{equation}
where
\begin{equation}\label{eq-}
e^{u(a)}R_\Omega(a)=R_{f(\Omega)}\big(f(a)\big)
\end{equation}
whenever $u=\ln|f'|$ for a conformal mapping $f$ from $\Omega$ onto
$f(\Omega)$.
\end{theorem}

\begin{proof} (i) For $t\ge 0$ and $a\in\Omega$ let
$$
\Omega_t=\{z\in\Omega:\ g_\Omega(z,a)>t\}.
$$
Then
$$
\partial\Omega_t=\{z\in\Omega:\ g_{\Omega}(z,a)=t\}.
$$
If $\psi$ is a conformal map from $\Omega$ onto the unit disk
$\mathbb D\subset\mathbb R^2$ with $\psi(a)=0$ then
$g_\Omega(\cdot,\cdot)$ can be expressed by the following formula
(cf. \cite[p. 172]{Du}):
$$
g_\Omega(z,a)=-(2\pi)^{-1}\ln |\psi(z)|,\quad z\in\Omega,
$$
and hence a direct computation yields
$$
\frac{\partial g_\Omega(z,a)}{\partial
n}=(2\pi)^{-1}{|\psi'(z)|},\quad z\in\Omega.
$$

In the above and below, ${\partial}/{\partial n}$ is the inner
normal derivative. Putting
$$
X(t)=\int_{\Omega_t}e^{2u(z)}dA(z) \quad\hbox{and}\quad
X_0(t)=\int_{\partial\Omega_t}\Big|\frac{e^{u(z)}}{\psi'(z)}\Big|^2\Big(\frac{\partial
g_\Omega(z,a)}{\partial n}\Big)dL(z),
$$
we get that for $p\in [0,\infty)$,
\begin{eqnarray*}
Y_{p}(t)&=&\int_{{\Omega}_t}\big(g_{\Omega}(z,a)\big)^{p}
e^{2u(z)}dA(z)\\
&=&\int_{{\Omega}_t}\Big|\frac{e^{u(z)}}{\psi'(z)}\Big|^2\big(g_{\Omega}(z,a)\big)^{p}
|\psi'(z)|^2dA(z)\\
&=&\int_{[t,\infty)}\int_{\partial\Omega_r}\big(g_{\Omega}(z,a)\big)^{p}\Big|\frac{e^{u(z)}}{\psi'(z)}\Big|^2\Big(\frac{\partial
g_\Omega(z,a)}{\partial n}\Big)^2
dL(z)dn\\
&=&\int_t^\infty\left(\int_{\partial\Omega_r}\big(g_{\Omega}(z,a)\big)^{p}\Big|\frac{e^{u(z)}}{\psi'(z)}\Big|^2\Big(\frac{\partial
g_\Omega(z,a)}{\partial n}
\Big) dL(z)\right)dr\\
&=&\int_t^\infty r^{p} X_0(r)dr,
\end{eqnarray*}
whence finding
\begin{equation}\label{eq10}
X(t)=\int_t^\infty X_0(r)dr,\quad t>0.
\end{equation}
The formula (\ref{eq10}) indicates that $X(\cdot)$ satisfies the
first condition of (\ref{eql}). Moreover, letting $t\to 0^+$ we
achieve
\begin{eqnarray*}
Y_{p}(0)&=&\int_\Omega\big(g_\Omega(z,a)\big)^{p}e^{2u(z)}dA(z)\\
&=&\int_0^\infty r^{p}X_0(r)dr\\
&=&-\int_0^\infty r^{p}dX(r).
\end{eqnarray*}

Since $\partial\Omega_t$ is a real analytic curve and
$$
\int_{\partial\Omega_t}\Big(\frac{\partial g_\Omega(z,a)}{\partial
n}\Big)dL(z)=(2\pi)^{-1}\int_{\partial\Omega_t}|\psi'(z)|dL(z)=1
$$
for almost all $t\ge 0$, we may apply Huber's inequality in Theorem
\ref{t3} to $\Omega_t\cup\partial\Omega_t$ and then use the
Cauchy-Schwarz inequality to deduce

\begin{eqnarray}\label{eq11}
4\pi\kappa(\Omega)
X(t)&\le&4\pi\Big(1-(2\pi)^{-1}\int_{\Omega_t}\max\{-\Delta u(z),0\}dA(z)\Big)\int_{\Omega_t}e^{2u(z)}dA(z)\nonumber\\
&\le&\left(\int_{\partial\Omega_t}e^{u(z)}dL(z)\right)^2\nonumber\\
&=&\left(\int_{\partial\Omega_t}\Big|\frac{e^{u(z)}}{\psi'(z)}\Big|\Big(\frac{\partial g_\Omega(z,a)}{\partial n}\Big)dL(z)\right)^2\\
&\le&\left(\int_{\partial\Omega_t}\Big|\frac{e^{u(z)}}{\psi'(z)}\Big|^2\Big(\frac{\partial
g_\Omega(z,a)}{\partial n}\Big)dL(z)\right)
\left(\int_{\partial\Omega_t}\Big(\frac{\partial
g_\Omega(z,a)}{\partial
n}\Big)dL(z)\right)\nonumber\\
&=&X_0(t)\nonumber.
\end{eqnarray}
The estimate (\ref{eq11}) ensures
$$
\frac{d\big(e^{4\pi\kappa(\Omega)
t}X(t)\big)}{dt}=e^{4\pi\kappa(\Omega) t}\big(4\pi\kappa(\Omega)
X(t)-X_0(t)\big)\le 0,
$$
and then makes the second condition in (\ref{eql}) available for
$c=4\pi\kappa(\Omega)$. An easy application of Theorem \ref{l6} (i)
implies that (\ref{eq6}) is true for all $a\in\Omega$.

Next, let us handle the equality of (\ref{eq6}). If $u$ equals
$\ln|f'|$ for a conformal map $f$ from $\Omega$ onto
$D(f(a),R)=\{w\in\mathbb{R}^2: |w-f(a)|<R\}$ for which the Green
function is
$$
g_{D(f(a),R)}(w_1,w_2)=(2\pi)^{-1}\ln\left|\frac{R^2-\overline{(w_1-f(a))}(w_2-f(a))}{R(w_1-w_2)}\right|
$$
where $w_1,w_2\in D(f(a),R)$, then $\kappa(\Omega)=1$, and hence
from the conformal invariance of Green's functions it follows that
for $t\ge 0$,
\begin{eqnarray*}
X(t)&=&\int_{\{z\in\Omega:\ g_{\Omega}(z,a)>t\}}|f'(z)|^2dA(z)\\
&=&\int_{\{w\in D(f(a),R):\ g_{D(f(a),R)}(w,f(a))>t\}}dA(w)\\
&=&\int_{\{w\in D(f(a),R):\ \ln (R/|w-f(a)|)>2\pi t\}}dA(w)\\
&=&e^{-4\pi t}\pi R^2.
\end{eqnarray*}
Accordingly, the equality part of Theorem \ref{l6} (i) is used to
derive the validity of the equality of (\ref{eq6}).

Conversely, the equality part of Theorem \ref{l6} (i) suggests us to
show only that $X(t)=e^{-4\pi\kappa(\Omega) t}X(0)$ implies $u=\ln
|f'|$ where $f$ is a conformal map from $\Omega$ onto a Euclidean
disk centered at $f(a)$ in $\mathbb R^2$. Now, suppose
$X(t)=e^{-4\pi\kappa(\Omega) t}X(0)$. By (\ref{eq10}) we have
$$
e^{-4\pi\kappa(\Omega) t}X(0)=\int_t^\infty
X_0(r)dr=e^{-4\pi\kappa(\Omega) t}\int_0^\infty X_0(r)dr.
$$
Differentiating the left-hand equality and using the hypothesis, we
obtain
$$
X_0(t)=4\pi\kappa(\Omega) e^{-4\pi\kappa(\Omega) t}X(0)=4\pi\kappa(\Omega)
X(t),\quad t\ge 0,
$$
and consequently, the first inequality in (\ref{eq11}) becomes an
equality for $t\ge 0$. According to the equality case of Theorem
\ref{t3}, we know that $u$ is the same as $\ln|f'|$ for a conformal
map $f$ from $\Omega$ onto a Euclidean disk $D(b,R)$ with center $b$
and radius $R$. So, $\kappa(\Omega)=1$. This in turn yields
\begin{equation}\label{eq11a}
X(t)=\int_{\{z\in\Omega: g_{\Omega}(z,a)>t\}}|f'(z)|^2dA(z)=e^{-4\pi
t}X(0).
\end{equation}
The formula (\ref{eq11a}) must enforce $b=f(a)$. To see this point,
suppose $b\not=f(a)$, then $\delta=\overline{f(a)-b}$ meets
$0<|\delta|<R$ and $0<\lambda=(R^2-|\delta|^2)/R<R$. Because of
(\ref{eq11a}) and the conformal invariance of Green's functions, we
get
\begin{eqnarray*}
e^{-4\pi t}\pi R^2&=&e^{-4\pi t}\int_{\{w\in D(b,R):\ g_{D(b,R)}(w,f(a))>0\}}dA(w)\\
&=&\int_{\{w\in D(b,R):\ g_{D(b,R)}(w,f(a))>t\}}dA(w)\\
&=&\int_{\{w\in D(b,R):\
|{\lambda}/{(w-f(a))}-{\delta}/{R}|>e^{2\pi t}\}}dA(w)\\
&\le&\lambda^2\pi e^{-4\pi t}<R^2\pi e^{-4\pi t}\quad\hbox{as}\quad
t\to\infty,
\end{eqnarray*}
thereby reaching a contradiction.

(ii) Owing to $X(0)=\int_{\Omega}e^{2u(z)}dA(z)<\infty$, the
preceding argument and (\ref{eq10a}) yield that
$$
\lim_{p\to\infty}\mathcal{F}\big(p,a,\kappa(\Omega)\big)=\lim_{t\to\infty}e^{4\pi\kappa(\Omega)
t}X(t)
$$
exists for every $a\in\Omega$. Fix a point $z_0\in\Omega$ and
suppose $f$ is a Riemann mapping associated with $z_0$ -- that is --
a conformal map $f$ from $\Omega$ onto the unit open disk $\mathbb
D$ such that $f(z_0)=0$ and $R_\Omega(z_0)=|f'(z_0)|^{-1}$. Via the
superposition $F(z,a)=\phi_{f(a)}\big(f(z)\big)$ of $f$ with a
standard M\"obius transform from $\mathbb D$ onto itself:
$$
\phi_{w}(z)=\frac{w-z}{1-\bar{w}z},\quad z,w\in\mathbb D,
$$
we find that $F(\cdot,a)$ is a Riemann mapping associated with
$a\in\Omega$, and so that
$$
R_\Omega(a)=\left|\frac{d
F(z,a)}{dz}\big|_{z=a}\right|^{-1}=\frac{1-|f(a)|^2}{|f'(a)|}.
$$
Now, if $h$ is the inverse map of $f$ and $a=h(b)$, then
\begin{eqnarray*}
X(t)&=&\int_{\Omega_t}e^{2u(z)}dA(z)\\
&=&\int_{\{w\in{\mathbb D}:\ g_{h(\mathbb
D)}(h(w),h(b))>t\}}e^{2u\circ h(w)}|h'(w)|^2 dA(w)\\
&=&\int_{\{w\in{\mathbb D}:\ |w|<e^{-2\pi t}\}}e^{2u\circ
f\circ\phi_b(w)}|(h\circ\phi_b)'(w)|^2 dA(w),
\end{eqnarray*}
and hence
\begin{eqnarray*}
\lim_{t\to\infty}\frac{X(t)}{e^{-4\pi\kappa(\Omega)
t}}&=&\pi\lim_{t\to\infty}\frac{e^{4\pi\big(\kappa(\Omega)-1\big)t}}{\pi
e^{-4\pi t}}\int_{\{w\in{\mathbb D}:\ e^{2\pi
t}|w|<1\}}\Big(\frac{|(h\circ\phi_b)'(w)|}{e^{-u\circ
f\circ\phi_b(w)}}\Big)^2 dA(w)\\
&=&\pi\lim_{t\to\infty}e^{4\pi\big(\kappa(\Omega)-1\big)t}e^{2u\circ
h(b)}|(f\circ\phi_b)'(0)|^2\\
&=&\pi\lim_{t\to\infty}e^{4\pi\big(\kappa(\Omega)-1\big)t}\big(e^{u(a)}|h'(b)|(1-|b|^2)\big)^2\\
&=&\pi\lim_{t\to\infty}e^{4\pi\big(\kappa(\Omega)-1\big)t}\Big(\frac{e^{u(a)}(1-|f(a)|^2)}{|f'(a)|}\Big)^2\\
&=&\pi\lim_{t\to\infty}e^{4\pi\big(\kappa(\Omega)-1\big)t}\big(e^{u(a)}R_\Omega(a)\big)^2.
\end{eqnarray*}

Now, if there is a conformal mapping $f$ from $\Omega$ onto
$f(\Omega)$ such that $e^u=|f'|$, then the conformal transformation
law for the Robin function/mass (cf. \cite{BaFl}) derives
$$
e^{u(a)}R_\Omega(a)=|f'(a)|R_\Omega(a)=R_{f(\Omega)}\big(f(a)\big),
$$
as desired.
\end{proof}

\begin{remark}\label{r31} In accordance with Remark \ref{r22} we strongly feel that
Theorem \ref{t4} (i) is also true for $-1<p_1<0$. The coming-up-next
estimates, corresponding to ones in Remark \ref{r22}, are in support
of this feeling.

\item{\rm(i)} When $-1<p_1<0$,
\begin{equation}\label{eqr31}
\int_{\Omega}e^{2u(z)}dA(z)\le\Big(\frac{\pi p_1}{\sin\pi
p_1}\Big)\Big(\frac{\big(4\pi\kappa(\Omega)\big)^{p_1}}{\Gamma(1+p_1)}\Big)\int_\Omega
\big(g_\Omega(z,a)\big)^{p_1}e^{2u(z)}dA(z).
\end{equation}
The inequality (\ref{eqr31}), along with H\"older's inequality,
gives that if $-1<p_1<p_2<0$ then
\begin{equation}\label{eqr32}
\int_{\Omega}\big(g_\Omega(z,a)\big)^{p_2}e^{2u(z)}dA(z)\le
c\big(p_1,p_2,\kappa(\Omega)\big)\int_{\Omega}\big(g_\Omega(z,a)\big)^{p_1}e^{2u(z)}dA(z),
\end{equation}
where
$$
c\big(p_1,p_2,\kappa(\Omega)\big)=\min\left\{\Big(\frac{\pi p_1
\big(4\pi\kappa(\Omega)\big)^{p_1}}{(\sin\pi
p_1)\Gamma(1+p_1)}\Big)^{1-\frac{p_2}{p_1}},\Big(\frac{\pi p_2
\big(4\pi\kappa(\Omega)\big)^{p_2}}{(\sin\pi
p_2)\Gamma(1+p_2)}\Big)^{\frac{p_2}{p_1}-1}\right\},
$$
and (\ref{eqr32}) is not optimal.

\item{\rm(ii)} When $-1<p<0$,

\begin{eqnarray}\label{eq++b}
\int_{\Omega}e^{2u(z)}dA(z)&\le&-\big(4\pi\kappa(\Omega)\big)^{-1}\lim_{t\to
0^+}\frac{d}{dt}\Big(\int_{\Omega_t}e^{2u(z)}dA(z)\Big)\nonumber\\
&=&\lim_{p\to-1^+}\frac{\big(4\pi\kappa(\Omega)\big)^{p}}{\Gamma(1+p)}
\int_\Omega\big(g_\Omega(z,a)\big)^pe^{2u(z)}dA(z)\\
&=&\big(4\pi\kappa(\Omega)\big)^{-1}
\int_{\partial\Omega}\left(\frac{e^{2u(z)}}{\frac{\partial
g_\Omega(z,a)}{\partial n}}\right)dL(z)\nonumber.
\end{eqnarray}

\item{\rm(iii)} From (\ref{eq10a}) and (\ref{eq+}) we see
\begin{equation}\label{eq+b}
\lim_{t\to\infty}e^{4\pi\kappa(\Omega)
t}\int_{\Omega_t}e^{2u(z)}dA(z)=\left\{\begin{array} {r@{\;,\quad}l}
0 & \kappa(\Omega)<1,\\
\pi\big(e^{u(a)}R_\Omega(a)\big)^2 & \kappa(\Omega)=1,
\end{array}
\right.
\end{equation}
whose special case $u=0$ produces the corresponding limit formula in
\cite[Lemma 10]{Flu} (cf. \cite[Lemma 1(c)]{Lin}).
\end{remark}

More interestingly, a combination of Theorems \ref{t3}-\ref{t4} and
Remark \ref{r31} implies a chain of inequalities linking the
integrals on a domain and its boundary.

\begin{corollary}\label{c5} Let $\Omega$ be the interior of a rectifiable simple curve in $\mathbb R^2$. Suppose $u$ is continuous on $\Omega\cup\partial\Omega$ and of class $C^1$
as well as its second-order derivatives are piecewise continuous on
$\Omega$. Suppose
$$
p\in (0,\infty),\quad a\in\Omega,\quad
\kappa(\Omega)=1-(2\pi)^{-1}\int_\Omega\max\{-\Delta
u(z),0\}dA(z)>0,
$$
and
$$
\mathcal{F}\big(p,a,\kappa(\Omega)\big)=\frac{\big(4\pi\kappa(\Omega)\big)^p}{\Gamma(p+1)}\int_{\Omega}\big(g_\Omega(z,a)\big)^pe^{2u(z)}
dA(z).
$$
Then
\begin{equation}\label{eq7}
4\pi\kappa(\Omega)\mathcal{F}\big(p,a,\kappa(\Omega)\big)\le\left(\int_{\partial\Omega}
e^{u(z)}dL(z)\right)^2\le\int_{\partial\Omega}\left(\frac{e^{2u(z)}}{\frac{\partial
g_\Omega(z,a)}{\partial n}}\right)dL(z),
\end{equation}
where the left- (right-) hand equality in (\ref{eq7}) occurs when
and only when there is a conformal map $f$ from $\Omega$ onto a
Euclidean disk centered at $f(a)$ in $\mathbb R^2$ such that
$u=\ln|f'|$ (there is a positive number $\lambda$ such that
$u=\ln\big(\lambda\partial g_\Omega(z,a)/\partial n\big)$).
\end{corollary}

\begin{proof} Since the setting $0=p_1<p_2=p<\infty$ of Theorem \ref{t4} (i) tells us
that
$$
\mathcal{F}\big(p,a,\kappa(\Omega)\big)\le
\int_{\Omega}e^{2u(z)}dA(z)
$$
holds for every $a\in\Omega$, the corollary follows from Theorems
\ref{t3} and \ref{t4}, the foregoing inequality and the following
Cauchy-Schwarz's inequality-based estimate:
\begin{eqnarray*}
\left(\int_{\partial\Omega}e^{u(z)}dL(z)\right)^2&\le&\int_{\partial\Omega}\left(\frac{e^{2u(z)}}{\frac{\partial
g_\Omega(z,a)}{\partial
n}}\right)dL(z)\int_{\partial\Omega}\left(\frac{\partial
g_\Omega(z,a)}{\partial
n}\right)dL(z)\\
&=&\int_{\partial\Omega}\left(\frac{e^{2u(z)}}{\frac{\partial
g_\Omega(z,a)}{\partial n}}\right)dL(z),
\end{eqnarray*}
where the inequality becomes an equality when and only when
$$
{e^{2u(z)}}{\left(\frac{\partial g_\Omega(z,a)}{\partial
n}\right)^{-1}}=\lambda\left(\frac{\partial g_\Omega(z,a)}{\partial
n}\right)
$$
holds for some constant $\lambda>0$.
\end{proof}

\section{Optimal Monotonicity -- Geometric Form}

The monotonicity established in the last section may be extendable
to an optimal geometric monotone principle for the $L^p$-integral of
Green's function of a simply-connected domain on any abstract
surface (cf. \cite{Ba} for more information).

To see this, suppose $S$ is a surface which has such an isothermic
representation $(\Sigma,\mathsf{\sigma})$ that $\Sigma$ is a
subdomain of $\mathbb R^2$ and has the positive definite quadratic
form (i.e., Riemannian metric):
$$
\mathsf{\sigma}=e^{2u}ds^2=e^{2u}|dz|^2=e^{2u(z)}(dx^2+dy^2),\quad
z=x+iy\in\Sigma.
$$
Of course, $u$ is here assumed to be continuous on $\Sigma$ and its
boundary $\partial\Sigma$, be of class $C^1$, and have piecewise
continuous second-order partial derivatives on $\Sigma$.

Under this parameter system the Gauss curvature at every point of
$(\Sigma,\mathsf{\sigma})$ is determined by
\begin{equation}\label{eq411a}
K_\mathsf{\sigma}=-e^{-2u}\Delta u=-\Delta_\mathsf{\sigma} u,
\end{equation}
where $\Delta_\mathsf{\sigma}$ is the Laplacian operator associated
with the planar conformal metric $\mathsf{\sigma}$. Here it is
perhaps appropriate to mention the following open problem of Berger
type: Find a conformal metric $\mathsf{\sigma}=e^{2u}ds^2$ on a
domain $\Sigma\subseteq\mathbb R^2$ with prescribed Gaussian
curvature $K$; equivalently find a solution $u$ to the semi-linear
elliptic equation $Ke^{2u}+\Delta u=0$ for a given function $K$ on
$\Sigma$. It is well-known that if $K=-4$ and $\Sigma=\Omega$
(considered in the last section) then $\Delta u=4e^{2u}$ is the
so-called Liouville's equation and takes the Robin function/mass
$H_\Omega(\cdot,\cdot)$ as the solution (see e.g. \cite{BaFl}).
Furthermore, it is proved in \cite{Sa} that if $K$ is of class $C^2$
and bounded on a bounded domain $\Sigma$ then the Liouville equation
has a solution on $\Sigma$. Additionally, on the unbounded domain
$\Sigma=\mathbb R^2$, searching for a solution of the equation under
the condition $\int_{\mathbb R^2}KdA_\mathsf{\sigma}<\infty$ is of
particular interest; see \cite{ChKi}, and \cite{KaW} (showing that
$K\in C^\infty(\mathbb R^2)$ is the Gauss curvature of some
Riemannian metric on $\mathbb R^2$), as well as \cite{Cha} (for more
information on nonlinear elliptic equations in conformal geometry).

Given a bounded and open subset $(O,\mathsf{\sigma})$ of
$(\Sigma,\mathsf{\sigma})$ with boundary $(\partial
O,\mathsf{\sigma})$, we denote by $g_{(O,\sigma)}(\cdot,a)$ the
Green function of $(O,\mathsf{\sigma})$ with pole $a\in O$ for
$\Delta_\mathsf{\sigma}$ provided that this function is determined
by the Dirichlet boundary problem:
\[
\left\{\begin{array} {r@{\quad,\quad}l}
\Delta_{\sigma} g_{(O,\sigma)}(z,a)=-\delta_a(z) & z\in O,\\
g_{(O,\sigma)}(z,a)=0 & z\in \partial O.
\end{array}
\right.
\]
Note that the first equation is understood under the distribution
with respect to the area element $dA_\mathsf{\sigma}$. So, this
Green function $g_{(O,\sigma)}(z,a)$ coincides with the Green
function $g_O(z,a)$ (i.e., $g_{(O,ds^2)}(z,a)$) for $\Delta$
discussed in the last section. Usually, the definition of the Green
function $g_{(O,\sigma)}(\cdot,\cdot)$ can be extended to the
surface $(\Sigma,\mathsf{\sigma})$ through setting
$g_{(O,\sigma)}(z,a)=0$ for $z\in \Sigma\setminus O$.

On the surface $(\Sigma,\mathsf{\sigma})$ the length and area
elements are defined by
$$
dL_\mathsf{\sigma}(z)=e^{u(z)}dL(z)\quad\hbox{and}\quad
dA_\mathsf{\sigma}(z)=e^{2u(z)}dA(z)\quad\hbox{for}\quad z\in\Sigma
$$
respectively. This gives the length of a rectifiable simple curve
$C=(\partial\Omega,\mathsf{\sigma})$ on $(\Sigma,\mathsf{\sigma})$
and the area of a simply-connected domain
$D=(\Omega,\mathsf{\sigma})$:
$$
L_\mathsf{\sigma}(C)=\int_{C}dL_\mathsf{\sigma}=\int_{\partial\Omega}e^{u(z)}dL(z)\quad\hbox{and}\quad
A_\mathsf{\sigma}(D)=\int_{D}dA_\mathsf{\sigma}=\int_{\Omega}e^{2u(z)}dA(z).
$$
As a result, the distance $d_\mathsf{\sigma}(z,a)$ between $z$ and
$a$ in $(\Sigma,\mathsf{\sigma})$ is defined by $\inf_\gamma
L_\mathsf{\sigma}(\gamma)$ where the infimum is taken over all
rectifiable simple curves $\gamma$ connecting $z$ and $a$. In terms
of the Green function and the distance function, we introduce a
concept of the Robin function/mass $H_{(\Omega,\sigma)}(a,a)$ and
the conformal radius $R_{(\Omega,\sigma)}(a)$ of
$(\Omega,\mathsf{\sigma})$ below:
$$
H_{(\Omega,\sigma)}(a,a)=-2\pi\lim_{z\to a}\Big((2\pi)^{-1}\ln
{d_\mathsf{\sigma}(z,a)}+g_{(\Omega,\sigma)}(z,a)\Big)
$$
and
$$
R_{(\Omega,\sigma)}(a)=\exp\big(-H_{(\Omega,\sigma)}(a,a)\big).
$$
Furthermore, let
$$
K^{\pm}_\mathsf{\sigma}(z)=\max\{\pm
K_\mathsf{\sigma}(z),0\}=\max\{\mp\Delta_\mathsf{\sigma} u(z),0\}.
$$
Then the surface version of the Huber's Theorem \ref{t3} is the
following assertion (cf. \cite[Theorem 3]{Hu54}).

\begin{theorem}\label{t4a} Let $\mathsf{\sigma}=e^{2u}ds^2$ be a conformal metric on a domain $\Sigma\subseteq\mathbb R^2$ for which
$u$ is continuous on $\Sigma\cup\partial\Sigma$ but also is of class
$C^1$ and piecewise continuous second-order partial derivatives on
$\Sigma$. If a rectifiable simple curve $\partial D$ of length
$L_\mathsf{\sigma}(\partial D)$ encloses a simply-connected domain
$D$ of area $A_\mathsf{\sigma}(D)$ on the surface
$(\Sigma,\mathsf{\sigma})$, then
\begin{equation}\label{eq4b}
\big(L_\mathsf{\sigma}(\partial D)\big)^2\ge 4\pi
A_\mathsf{\sigma}(D)\Big(1-(2\pi)^{-1}{\int_D
K_\mathsf{\sigma}^+dA_\mathsf{\sigma}}\Big).
\end{equation}
The equality in (\ref{eq4b}) holds when and only when
$K_\mathsf{\sigma}$ vanishes on $D$ and $\partial D$ is a geodesic
circle on $(\Sigma,\mathsf{\sigma})$.
\end{theorem}

With the help of Theorems \ref{l6} and \ref{t4a}, we obtain a
geometric description of Theorem \ref{t4}.

\begin{theorem}\label{t4b} Let $\mathsf{\sigma}=e^{2u}ds^2$ be a conformal metric on a domain $\Sigma\subseteq\mathbb R^2$
for which $u$ is continuous on $\Sigma\cup\partial\Sigma$ but also
is of class $C^1$ and piecewise continuous second-order partial
derivatives on $\Sigma$. Suppose $D=(\Omega,\mathsf{\sigma})$ is a
simply-connected domain with $\partial
D=(\partial\Omega,\mathsf{\sigma})$ being a rectifiable simple curve
on $(\Sigma,ds^2)$. If
$$
p\in [0,\infty),\quad (a,\sigma)\in D,\quad
\kappa_\sigma(D)=1-(2\pi)^{-1}\int_D
K_\mathsf{\sigma}^+dA_\mathsf{\sigma}>0,
$$
and
$$
\mathcal{G}\big(p,a,\kappa_\sigma(D)\big)=\frac{\big(4\pi\kappa_\sigma(D)\big)^p}{\Gamma(p+1)}\int_{D}\big(g_{D}(\cdot,a)\big)^p
dA_\mathsf{\sigma}(\cdot),
$$
then

\item{\rm(i)}
\begin{equation}\label{eq4111a}
0\le
p_1<p_2<\infty\Rightarrow\mathcal{G}\big(p_2,a,\kappa_\sigma(D)\big)\le\mathcal{G}\big(p_1,a,\kappa_\sigma(D)\big),
\end{equation}
where the right-hand equality in (\ref{eq4111a}) occurs when and
only when $K_\mathsf{\sigma}$ vanishes on $D$ and $\partial D$ is a
geodesic circle centered at $(a,\mathsf{\sigma})\in D$.

\item{\rm(ii)}
\begin{equation}\label{eq4111b}
\lim_{p\to\infty}\mathcal{G}(p,a,\kappa_\sigma(D))=\left\{\begin{array}
{r@{\;,\quad}l}
0 & \kappa_\sigma(D)<1,\\
\pi \big(R_{(\Omega,\sigma)}(a)\big)^2 & \kappa_\sigma(D)=1,
\end{array}
\right.
\end{equation}
where
\begin{equation}\label{eq4111c}
R_{(\Omega,\sigma)}(a)=R_{f(\Omega)}\big(f(a)\big)
\end{equation}
whenever $u=\ln|f'|$ for a conformal mapping $f$ from $\Omega$ onto
$f(\Omega)$.
\end{theorem}

\begin{proof} Since $\mathcal{G}(p,a,\kappa_\sigma(D))$ actually coincides with $\mathcal{F}(p,a,\kappa(\Omega))$, (\ref{eq4111a}) follows from (\ref{eq6})
right away. Moreover, the right-hand equality in (\ref{eq4111a})
holds if and only if the right-hand equality in (\ref{eq6}) holds.
This amounts to $u=\ln|f'|$ where $w=f(z)$ is a conformal mapping
from $\Omega$ onto a Euclidean disk centered at $f(a)$ in $\mathbb
R^2$. Note that for such a conformal mapping $f$,
$$
|dw|=|f'(z)||dz|=e^u ds=dL_\mathsf{\sigma}.
$$
Thus we see that there is an isometry from $D$ onto a Euclidean disk
centered at $f(a)$, thereby getting that $K_\mathsf{\sigma}=0$ on
$D$ but also the boundary $\partial D$ becomes a geodesic circle
with center $(a,\mathsf{\sigma})$.

Next, (\ref{eq4111b}) and (\ref{eq4111c}) follow from (\ref{eq+}),
(\ref{eq-}) and a series of calculations:
\begin{eqnarray*}
(2\pi)^{-1}{\ln \big(R_{(\Omega,\sigma)}(a)\big)}&=&\lim_{z\to
a}\Big((2\pi)^{-1}{\ln
d_\mathsf{\sigma}(z,a)}+g_{(\Omega,\sigma)}(z,a)\Big)\\
&=&\lim_{z\to a}\Big((2\pi)^{-1}{\ln
d_\mathsf{\sigma}(z,a)}+g_{\Omega}(z,a)\Big)\\
&=&\lim_{z\to a}\Big((2\pi)^{-1}{\ln
\big(e^{u(a)}|z-a|\big)}+g_{\Omega}(z,a)+\mathcal{O}(|z-a|)\Big)\\
&=&(2\pi)^{-1}\big({u(a)}+\ln R_\Omega(a)\big).
\end{eqnarray*}
In the last second equality we have used a readily-checked fact (cf.
\cite[Lemma 1]{Ste}) that there are two positive constants $c_1,c_2$
to ensure the implication:
$$
|z-a|<c_1\Rightarrow\left|\ln
\frac{d_\mathsf{\sigma}(z,a)}{|z-a|}-u(a)\right|\le c_2|z-a|.
$$

\end{proof}

\begin{remark}\label{r41} Like Remark \ref{r31}, we have (\ref{eqr31a})-(\ref{eqr32a})-(\ref{eqr31b}) parallel to
(\ref{eqr31})-(\ref{eqr32})-(\ref{eq++b}):

\item{\rm(i)} When $-1<p_1<0$,
\begin{equation}\label{eqr31a}
A_\mathsf{\sigma}(D)\le\Big(\frac{\pi p_1}{\sin\pi
p_1}\Big)\Big(\frac{(4\pi\kappa_\sigma(D))^{p_1}}{\Gamma(1+p_1)}\Big)\int_{\Omega}
\big(g_{(\Omega,\sigma)}(\cdot,a)\big)^{p_1}dA_\mathsf{\sigma}(\cdot).
\end{equation}
The inequality (\ref{eqr31}), plus H\"older's inequality, gives that
for $-1<p_1<p_2<0$,
\begin{equation}\label{eqr32a}
\int_{\Omega}\big(g_{(\Omega,\sigma)}(\cdot,a)\big)^{p_2}dA_\mathsf{\sigma}(\cdot)\le
c(p_1,p_2,\kappa_\sigma(D))\int_{\Omega}\big(g_{(\Omega,\sigma)}(\cdot,a)\big)^{p_1}dA_\mathsf{\sigma}(\cdot).
\end{equation}

\item{\rm(ii)} When $-1<p<0$,

\begin{eqnarray}\label{eqr31b}
A_\mathsf{\sigma}(D)&\le&-\big(4\pi\kappa_\sigma(D)\big)^{-1}\lim_{t\to
0^+}\frac{d}{dt}\Big(\int_{\Omega_t}dA_\mathsf{\sigma}(\cdot)\Big)\nonumber\\
&=&\lim_{p\to-1^+}\frac{\big(4\pi\kappa_\sigma(D)\big)^{p}}{\Gamma(1+p)}
\int_D\big(g_{(\Omega,\sigma)}(\cdot,a)\big)^pdA_\mathsf{\sigma}(\cdot)\\
&=&\big(4\pi\kappa_\sigma(D)\big)^{-1} \int_{\partial
\Omega}\Big({\frac{\partial g_{(\Omega,\sigma)}(\cdot,a)}{\partial
n_\mathsf{\sigma}}}\Big)^{-1}dL_\mathsf{\sigma}(\cdot)\nonumber.
\end{eqnarray}
In the above and below,
$$
\frac{\partial g_{(\Omega,\sigma)}(\cdot,a)}{\partial
n_\mathsf{\sigma}}=e^{-u(\cdot)}\frac{\partial
g_{\Omega}(\cdot,a)}{\partial n}
$$
is the inner normal derivative of $g_{(\Omega,\sigma)}(\cdot,a)$
with respect to the metric $\mathsf{\sigma}=e^{2u}ds^2$.

\item{\rm(iii)} From (\ref{eq10a}) and (\ref{eq4111b}) we see the
counterpart of (\ref{eq+b}) below:
\begin{equation}\label{eq+a}
\lim_{t\to\infty}\frac{A_\sigma(\{z\in\Omega:
g_{(\Omega,\sigma)}(z,a)>t\}\big)}{e^{-4\pi\kappa_\sigma(D)
t}}=\left\{\begin{array} {r@{\;,\quad}l}
0 & \kappa_\sigma(D)<1,\\
\pi\big(R_{(\Omega,\sigma)}(a)\big)^2 & \kappa_\sigma(D)=1.
\end{array}
\right.
\end{equation}
\end{remark}

Needless to say, the newfound optimal isoperimetric-type inequality
in the following corollary is of particular interest.

\begin{corollary}\label{c45} Let $\mathsf{\sigma}=e^{2u}ds^2$ be a conformal metric on a domain $\Sigma\subseteq\mathbb R^2$
for which $u$ is continuous on $\Sigma\cup\partial\Sigma$ but also
is of class $C^1$ and piecewise continuous second-order partial
derivatives on $\Sigma$. Suppose $D=(\Omega,\mathsf{\sigma})$ is a
simply-connected domain on $(\Sigma,\mathsf{\sigma})$ with $\partial
D=(\partial\Omega,\mathsf{\sigma})$ being a rectifiable simple
curve. If
$$
p\in (0,\infty),\quad (a,\sigma)\in D,\quad
\kappa_\sigma(D)=1-(2\pi)^{-1}\int_D
K_\mathsf{\sigma}^+dA_\mathsf{\sigma}>0,
$$
and
$$
\mathcal{G}\big(p,a,\kappa_\sigma(D)\big)=\frac{\big(4\pi\kappa_\sigma(D)\big)^p}{\Gamma(p+1)}
\int_{D}\big(g_{D}(\cdot,a)\big)^p dA_\mathsf{\sigma}(\cdot),
$$
then
\begin{equation}\label{eq7a}
4\pi\kappa_\sigma(D)\mathcal{G}\big(p,a,\kappa_\sigma(D)\big)\le
\big(L_\mathsf{\sigma}(\partial D)\big)^2\le\int_{\partial
D}\Big({\frac{\partial g_{D}(\cdot,a)}{\partial
n_\mathsf{\sigma}}}\Big)^{-1}dL_\mathsf{\sigma}(\cdot),
\end{equation}
where the left- (right-) hand inequality in (\ref{eq7a}) happens
when and only when $K_\mathsf{\sigma}$ vanishes on $D$ and $\partial
D$ is a geodesic circle centered at $(a,\mathsf{\sigma})\in D$
(there is a positive number $\lambda$ such that
$u=\ln\big(\lambda\partial g_{\Omega}(\cdot,a)/\partial n\big)$).
\end{corollary}
\begin{proof} This follows immediately Corollary \ref{c5}.
\end{proof}

\section{Application}

In this final section we are concerned about how to apply the
previous ideas, methods and techniques to settling some problems on
complete noncompact surfaces without boundary.

In accordance with the definition adapted by \cite{Li} and \cite{LiTa}, we say that $(\mathbb M^2,\mathsf{g})$ is a complete
noncompact boundary-free surface provided that $(\mathbb
M^2,\mathsf{g})$ is a two-dimensional complete noncompact manifold
$\mathbb M^2$ without boundary, equipped with a Riemannian metric
$\mathsf g$. On such a surface, we always employ
$$
d_\mathsf{g}(\cdot,\cdot);\ \ K_{\mathsf g}(\cdot);\ \
K^{\pm}_{\mathsf g}(\cdot)=\max\{\pm K_\mathsf{g},0\};\ \
\chi(\cdot);\ \ dA_{\mathsf g}(\cdot);\quad dL_{\mathsf g}(\cdot);\
\ \Delta_{\mathsf g}(\cdot);\ \ \nabla_{\mathsf g}(\cdot),
$$
to denote the distance function; the Gauss curvature; the positive
or negative part of the Gauss curvature; the Euler characteristic;
the area element; the length element; the Laplacian operator; the
gradient, respectively -- see also Shiohama-Shioya-Tanaka's
monograph \cite{ShShTa} for some related materials. The following
celebrated Gauss-Bonnet type results (i) and (ii) are due to
Cohn-Vossen \cite{CoVo} and Huber \cite{Hu1}, and Hartman \cite{Ha}
and Shiohama \cite{Sh}, in the above-mentioned order.

\begin{theorem}\label{t5a} Let $(\mathbb M^2,\mathsf{g})$ be a
complete noncompact boundary-free surface with $K^-_{\mathsf g}$
being integrable with respect to $dA_{\mathsf g}$. Then

\item{\rm(i)} $\mathbb M^2$ is conformally equivalent to a compact Riemann
surface minus finitely many points. Moreover
$$
\int_{\mathbb M^2}K_{\mathsf g}dA_{\mathsf g}\le 2\pi\chi(\mathbb
M^2)\quad\hbox{and}\quad\int_{\mathbb M^2}|K_{\mathsf g}|dA_{\mathsf
g}<\infty.
$$
Especially, $\mathbb M^2$ is conformally equivalent to $\mathbb R^2$
whenever $\mathbb M^2$ is simply-connected.

\item{\rm(ii)} For any geodesic ball $B(a,r)=\{z\in\mathbb M^2: d_\mathsf{g}(z,a)<r\}$ centered at
$a\in\mathbb M^2$ with radius $r>0$ and its boundary $\partial
B(a,r)=\{z\in\mathbb M^2: d_\mathsf{g}(z,a)=r\}$ on $(\mathbb
M^2,\mathsf{g})$,
$$
\chi(\mathbb M^2)-(2\pi)^{-1}\int_{\mathbb M^2}K_{\mathsf
g}dA_{\mathsf g}=\lim_{r\to\infty}\frac{\Big(L_{\mathsf{g}}(\partial
B(a,r)\big)\Big)^2}{4\pi A_{\mathsf{g}}\big(B(a,r)\big)}.
$$
\end{theorem}

Given a bounded and open subset $O$ of $\mathbb M^2$ with boundary
$\partial O$, we denote by $g_{(O,\mathsf{g})}(\cdot,a)$ the Green
function of $O$ with pole at $a\in O$ for $\Delta_{\mathsf{g}}$
provided this function is decided by the Dirichlet boundary problem:
\[
\left\{\begin{array} {r@{\quad,\quad}l}
\Delta_{\mathsf{g}} g_{(O,\mathsf{g})}(z,a)=-\delta_a(z) & z\in O,\\
g_{(O,\mathsf{g})}(z,a)=0 & z\in \partial O.
\end{array}
\right.
\]
The first equation is clearly understood under the distribution with
respect to the area element $dA_{\mathsf{g}}$. Moreover, the
definition of this Green's function can be extended to the surface
$(\mathbb M^2,\mathsf{g})$ via letting $g_{(O,\mathsf{g})}(z,a)=0$
for $z\in \mathbb M^2\setminus O$. From \cite[Theorem 4.13]{Au} it
turns out that there exist a small number $\epsilon>0$ and a
function $H_{(O,\mathsf{g})}(\cdot,\cdot)$ (which is continuous
symmetric on $O\times O$ and $C^\infty$-smooth on $O\times
O\setminus \{(a,a)\}$) such that $d_\mathsf{g}(z,a)<\epsilon$
implies
$$
g_{(O,\mathsf{g})}(z,a)=-(2\pi)^{-1}\big(\ln
d_\mathsf{g}(z,a)+H_{(O,\mathsf{g})}(z,a)\big).
$$
Consequently, a combined use of the Green function and the distance
function induces the Robin function/mass $H_{(O,\mathsf{g})}(a,a)$
and the conformal radius $R_{(O,\mathsf{g})}(a)$ at $a\in O$ under
the metric $\mathsf{g}$:
$$
H_{(O,\mathsf{g})}(a,a)=-2\pi\lim_{z\to a}\Big((2\pi)^{-1}\ln
{d_\mathsf{g}(z,a)}+g_{(O,\mathsf{g})}(z,a)\Big)
$$
and
$$
R_{(O,\mathsf{g})}(a)=\exp\big(-H_{(O,\mathsf{g})}(a,a)\big).
$$

As an immediate application of Theorems \ref{t3}-\ref{t4}, we have
the following assertion whose (i) has slightly stronger hypothesis
and conclusion than Li-Tam's ones in \cite[Theorem 5.1]{LiTa}.

\begin{theorem}\label{t51} Let $(\mathbb M^2,\mathsf{g})$ be a
simply-connected complete noncompact boundary-free surface with
$$
\int_{\mathbb M^2}K^-_gdA_\mathsf{g}<\infty\quad\hbox{and}\quad
\int_{\mathbb M^2}K^+_gddA_\mathsf{g}<2\pi.
$$
Then

\item{\rm(i)} For $(\mathbb M^2,\mathsf{g})$, the best isoperimetric
constant:
$$
\tau_g(\mathbb M^2)=\inf\left\{\frac{\Big(L_{\mathsf{g}}(\partial
O\big)\Big)^2}{4\pi A_{\mathsf{g}}\big(O\big)}:\quad O\in
BRD(\mathbb M^2)\right\}
$$
satisfies
\begin{equation}\label{eq5a}
1-(2\pi)^{-1}\int_{\mathbb M^2}K^+_{\mathsf
g}dA_\mathsf{g}\le\tau_g(\mathbb M^2)\le 1-(2\pi)^{-1}\int_{\mathbb
M^2}K_\mathsf{g}dA_\mathsf{g},
\end{equation}
where the infimum is taken over all relatively compact domains
$O\subseteq\mathbb M^2$ (written as $O\in RCD(\mathbb M^2)$).
Obviously, the equalities in (\ref{eq5a}) occur when $K_\mathsf{g}$
is nonnegative on $\mathbb M^2$.

\item{\rm(ii)} For $a\in O$, $O\in BRD(\mathbb M^2)$ with $C^\infty$ boundary $\partial O$, ${\partial g_{(O,\mathsf{g})}(\cdot,a)}/{\partial
n_\mathsf{g}}$-the inner normal derivative of
$g_{(O,\mathsf{g})}(\cdot,a)$ under $\mathsf{g}$, and $0\le
p<\infty$, the $L^p$-integral of the Green's function
$g_{(O,\mathsf{g})}(\cdot,a)$:
$$
\mathcal{H}\big(p,a,O,\tau_\mathsf{g}(\mathbb
M^2)\big)=\frac{\big(4\pi\tau_{\mathsf g}(\mathbb
M^2)\big)^{p}}{\Gamma(1+p)}\int_O\big(g_{(O,\mathsf{g})}(\cdot,a)\big)^{p}dA_{\mathsf
g}(\cdot)
$$
enjoys
\begin{eqnarray}\label{eq5b}
0\le
p_1<p_2<\infty&\Rightarrow&\mathcal{H}\big(p_2,a,O,\tau_\mathsf{g}(\mathbb
M^2)\big)\nonumber\\
&\le&\mathcal{H}\big(p_1,a,O,\tau_\mathsf{g}(\mathbb
M^2)\big)\nonumber\\
&\le&\big(4\pi\tau_\mathsf{g}(\mathbb
M^2)\big)^{-1}\big(L_\mathsf{g}(\partial O)\big)^2\\
&\le&\big(4\pi\tau_\mathsf{g}(\mathbb M^2)\big)^{-1}\int_{\partial
O}\Big(\frac{g_{(O,\mathsf{g})}(z,a)}{\partial
n_\mathsf{g}}\Big)^{-1}dL_\mathsf{g}(z),\nonumber
\end{eqnarray}
where the second/third/fourth equality in (\ref{eq5b}) holds when
$K_{\mathsf g}$ vanishes on $O$ but also $O$ is a geodesic ball
$B(a,r)$. Moreover,

\begin{equation}\label{eq5bb}
\lim_{p\to\infty}\mathcal{H}\big(p,a,O,\tau_\mathsf{g}(\mathbb
M^2)\big)=\lim_{t\to\infty}e^{4\pi t\tau_\mathsf{g}(\mathbb
M^2)}A_\mathsf{g}\big(\{z\in O:\ g_{(O,\mathsf{g})}(z,a)>t\}\big).
\end{equation}
In particular, if $K_\mathsf{g}\ge 0$ then
\begin{equation}\label{eq+bbb}
\lim_{p\to\infty}\mathcal{H}\big(p,a,O,\tau_\mathsf{g}(\mathbb
M^2)\big)=\left\{\begin{array} {r@{\;,\quad}l}
0 & \tau_\mathsf{g}(\mathbb M^2)<1,\\
\pi\big(R_{(O,\mathsf{g})}(a)\big)^2 & \tau_\mathsf{g}(\mathbb
M^2)=1.
\end{array}
\right.
\end{equation}

\end{theorem}
\begin{proof} Theorem \ref{t5a} (i) tells us that
$(\mathbb M^2,\mathsf{g})$ is homeomorphic to $(\mathbb
R^2,e^{2u}ds^2)$ where $u$ is of class $C^\infty$ on $\mathbb R^2$.
Thus we may consider $(\mathbb M^2,\mathsf{g})$ to be $(\mathbb
R^2,e^{2u}ds^2)$.

(i) Under this circumstance, any $O\in BRD(\mathbb M^2)$ may be
treated as a domain of the form $O=O_0\setminus(\cup_{j+1}^k D_j)$,
where $O_0$ is a simply-connected domain and contains mutually
disjoint simply-connected domains $O_1,...,O_k$ each of which is
homeomorphic to the unit disk $\mathbb D$. Using Theorem \ref{t4a}
we obtain
\begin{eqnarray*}
A_{\mathsf g}(O)&\le& A_{\mathsf g}(O_0)\\
&\le& \Big(4\pi\big(1-(2\pi)^{-1}\int_{O_0}K^+_\mathsf{g}dA_{\mathsf
g}\Big)^{-1} \big(L_{\mathsf{g}}(\partial O_0)\big)^2\\
&\le&\Big(4\pi\big(1-(2\pi)^{-1}\int_{\mathbb
M^2}K^+_\mathsf{g}dA_{\mathsf g}\Big)^{-1}
\big(L_{\mathsf{g}}(\partial O)\big)^2,
\end{eqnarray*}
whence verifying the left-hand inequality in (\ref{eq5a}). Clearly,
the right-hand inequality of (\ref{eq5a}) follows readily from
$$
\tau_\mathsf{g}(\mathbb M^2)\le\frac{\Big(L_{\mathsf{g}}(\partial
B(a,r)\big)\Big)^2}{4\pi A_{\mathsf{g}}\big(B(a,r)\big)}
$$
and Theorem \ref{t51} (ii) thanks to $\chi(\mathbb M^2)=1$ for the
simply-connected surface $\mathbb M^2$ and $B(a,r)\in BRD(\mathbb
M^2)$.

(ii) At this time, no conformal mapping is taken into account; yet
Theorem \ref{l6} and the key idea proving Theorem \ref{t4} will be
used. To do so, assume $a\in O$ and $O\in BRD(\mathbb M^2)$ with
$C^\infty$ boundary $\partial O$. For $t\ge 0$ set
$$
O_t=\big\{z\in O:\ g_{(O,\mathsf{g})}(z,a)>t\big\}.
$$
Then $g_{(O,\mathsf{g})}(\cdot,a)$ is of class $C^\infty$ on
$O\setminus\{a\}$, and for almost all $t>0$ one has
$$
\partial O_t=\{z\in O:\ g_{(O,\mathsf{g})}(z,a)=t\}.
$$
In the sequel, by $A_{\mathsf{g}}(O_t)$ we mean $\int_{O_t}
dA_\mathsf{g}$ for $t\ge 0$. As a function of $t$,
$A_{\mathsf{g}}(O_t)$ is decreasing and satisfied with the
differential formula
\begin{equation}\label{eq33a}
-\frac{dA_{\mathsf{g}}(O_t)}{dt}=\int_{\partial
O_t}\Big({\frac{\partial g_{(O,\mathsf{g})}(z,a)}{\partial
n_\mathsf{g}}}\Big)^{-1}dL_\mathsf{g}(z)\ge 0.
\end{equation}
Using the Cauchy-Schwarz inequality, (\ref{eq5a}), (\ref{eq33a}) and
the easily-verified formula (through \cite[p. 112, (22)]{Au} for
example)
\begin{equation}\label{eq5im}
\int_{\partial O_t}\left(\frac{{\partial
g_{(O,\mathsf{g})}(z,a)}}{\partial
n_\mathsf{g}}\right)dL_{\mathsf{g}}(z)=1,
\end{equation}
we get that for almost every $t>0$,
\begin{eqnarray*}
\Big(-\frac{dA_\mathsf{g}(O_t)}{dt}\Big)^\frac{1}2&=&\left(\int_{\partial
O_t}\frac{dL_{\mathsf{g}}(z)}{\frac{\partial
g_{(O,\mathsf{g})}(z,a)}{\partial n_\mathsf{g}}}\right)^\frac{1}2
\left(\int_{\partial O_t}{\Big(\frac{\partial
g_{(O,\mathsf{g})}(z,a)}{\partial
n_\mathsf{g}}\Big)}dL_{\mathsf{g}}(z)
\right)^\frac{1}2\\
&\ge&\int_{\partial O_t}dL_\mathsf{g}=L_\mathsf{g}(\partial O_t)\\
&\ge&\big(4\pi\tau_{\mathsf{g}}(\mathbb
M^2)\big)^\frac{1}{2}\big(A_{\mathsf{g}}(O_t)\big)^\frac{1}{2}.
\end{eqnarray*}
These equalities and inequalities yield
$$
\frac{d}{dt}\Big(\exp\big(4\pi\tau_{\mathsf{g}}(\mathbb
M^2)t\big)A_\mathsf{g}(O_t)\Big)=\frac{4\pi\tau_{\mathsf{g}}(\mathbb
M^2)A_\mathsf{g}(O_t)+\frac{dA_\mathsf{g}(O_t)}{dt}}{\exp\big(-4\pi\tau_{\mathsf{g}}(\mathbb
M^2)t\big)}\le 0.
$$
Note that if
$$
X(t)=A_\mathsf{g}(O_t);\quad Y_p(t)=-\int_t^\infty r^p
dA_\mathsf{g}(O_r)\ \ \hbox{for}\ \ p\in [0,\infty);\quad
c=4\pi\tau_\mathsf{g}(\mathbb M^2)
$$
then by the layer cake representation (cf. \cite[p. 26, Theorem
1.13]{LiLo}) and the integration-by-part,
$$
Y_p(t)=\int_{O_t}\big(g_{(O,\mathsf{g})}(z,a)\big)^p
dA_\mathsf{g}(z)=-\int_t^\infty r^p dX(r).
$$
Therefore, using Theorems \ref{l6}(i)-\ref{t4b}(i)-\ref{t51}(i) as
well as (\ref{eq+a}) we derive (\ref{eq5b}) and its equality case
whenever $0\le p_1<p_2<\infty$, as well as (\ref{eq5bb}) and
(\ref{eq+bbb}).
\end{proof}

Evidently, we can obtain the estimates similar to ones in Remark
\ref{r41} -- the details are left to the interested readers.
However, an important observation about the above argument is that
on a complete noncompact boundary-free surface the sharp
isoperimetric inequality must imply the optimal monotone principle
for the $L^p$-integral of Green's function. On the other hand,
according to the well-known Federer-Fleming type theorem for
$(\mathbb M^2,{\mathsf g})$, the isoperimetric inequality
\begin{equation}\label{eq5c-}
4\pi\tau_{\mathsf{g}}(\mathbb M^2)A_\mathsf{g}(O)\le
L_\mathsf{g}(\partial O)\quad\hbox{for}\ \ O\in BRD(\mathbb M^2)\ \
\hbox{with}\ \ C^\infty\ \ \hbox{boundary}\ \ \partial O,
\end{equation}
is equivalent to the Sobolev inequality
\begin{equation}\label{eq5c}
4\pi\tau_{\mathsf{g}}(\mathbb M^2)\int_{\mathbb
M^2}|f|^2dA_\mathsf{g}\le\Big(\int_{\mathbb M^2}|\nabla_\mathsf{g}
f|dA_\mathsf{g}\Big)^2\quad\hbox{for}\ \ f\in C^\infty_0(\mathbb
M^2),
\end{equation}
where $C^\infty_0(\mathbb M^2)$ represents the class of all
$C^\infty$ functions with compact support in $\mathbb M^2$. In
particular, if $K_\mathsf{g}\ge 0$ and (\ref{eq5c-})/(\ref{eq5c})
holds for $\tau_{\mathsf{g}}(\mathbb M^2)=1$ then $(\mathbb
M^2,\mathsf{g})$ is isometric to $(\mathbb R^2,ds^2)$ (cf. \cite[p.
244]{He}). Thus, a very natural question is ``What is an equivalent
analytic representation of the monotonicity for the $L^p$-integral
of Green's function?". Surprisingly but also naturally, the answer
to this question is related to both the Moser-Trudinger inequality
and the Nash-Sobolev inequality on $(\mathbb M^2,\mathsf{g})$.

\begin{theorem}\label{t52} Let $(\mathbb M^2,\mathsf{g})$ be a complete, noncompact, and boundary-free surface. Then the following implications (i)$\Rightarrow$(ii)$\Rightarrow$(iii)
are valid:

\item{\rm(i)} There are two positive constants $c_1$ and $C_1$ such
that Moser-Trudinger's inequality
\begin{equation}\label{eq5d-}
\int_O\exp\big(c_1|f(z)|^2\big)dA_\mathsf{g}(z)\le C_1
A_\mathsf{g}(O)
\end{equation}
holds for all $O\in BRD(\mathbb M^2)$ with $C^\infty$ boundary and
all
$$
f\in C^\infty_0(O)\quad\hbox{with}\quad
\int_O|\nabla_\mathsf{g}f|^2dA_\mathsf{g}\le 1.
$$

\item{\rm(ii)} There is a constant $C_2>0$ such that the $0=p_1<p_2=1$ monotonicity of Green's function
integral
\begin{equation}\label{eq5d}
\int_{O}g_{(O,\mathsf{g})}(z,a)dA_\mathsf{g}(z)\le
C_2A_\mathsf{g}(O),\quad a\in O
\end{equation}
holds for all $O\in BRD(\mathbb M^2)$ with $C^\infty$ boundary.

\item{\rm(iii)} There is a constant $C_3>0$ such that Nash-Sobolev's inequality
\begin{equation}\label{eq5f}
\Big(\int_{\mathbb M^2}|f|^2dA_\mathsf{g}\Big)^2\le
C_3\Big(\int_{\mathbb M^2}|\nabla_\mathsf{g}
f|^2dA_\mathsf{g}\Big)\Big(\int_{\mathbb M^2}|f|dA_\mathsf{g}\Big)^2
\end{equation}
holds for all $f\in C^\infty_0(\mathbb M^2)$.

Moreover, if there are two positive constants $c_0$ and $C_0$ such
that for any $a\in O$ and $O\in BRD(\mathbb M^2)$ with $C^\infty$
boundary one has
\begin{equation}\label{eq5ge}
A_\mathsf{g}\big(B(a,r)\big)\ge c_0r^2\quad\hbox{and}\quad
L_\mathsf{g}\big(\partial B(a,r)\big)\le C_0r\quad\hbox{for}\quad
0<r<\infty,
\end{equation}
then the implication (iii)$\Rightarrow$(ii) is valid too.
\end{theorem}
\begin{proof} (i)$\Rightarrow$(ii) Suppose (i) is true. For $t>0$, $a\in O$ and $O\in BRD(\mathbb M^2)$ with $C^\infty$ boundary, choose
$f_t(z)=\min\{g_{(O,\mathsf{g})}(z,a),t\}$ and set $Q_t=\{z\in O:\
g_{(O,\mathsf{g})}(z,a)<t\}$. Then by Green's formula and the
identity (\ref{eq5im}),
\begin{eqnarray*}
\int_O|\nabla_\mathsf{g}f_t|^2dA_\mathsf{g}&=&\int_{Q_t}|\nabla_\mathsf{g}
g_{(O,\mathsf{g})}(z,a)|^2dA_\mathsf{g}(z)\\
&=&\int_{Q_t}\big(\Delta_\mathsf{g}
g_{(O,\mathsf{g})}(z,a)\big)g_{(O,\mathsf{g})}(z,a)dA_\mathsf{g}(z)\\
&&+\ \ t\int_{\{z\in O:\
g_{(O,\mathsf{g})}(z,a)=t\}}\Big(\frac{\partial
g_{(O,\mathsf{g})}(z,a)}{\partial
n_\mathsf{g}}\Big)dL_\mathsf{g}(z)\\
&=& t.
\end{eqnarray*}
Meanwhile, we have
\begin{eqnarray*}
\int_O\exp\Big(c_1\big|{f_t}/{\sqrt{t}}\big|^2\Big)dA_\mathsf{g}(z)
&\ge&\int_{O\setminus Q_t}\exp\Big(c_1\big|{f_t}/{\sqrt{t}}\big|^2\Big)dA_\mathsf{g}(z)\\
&\ge&\exp(c_1t)A_\mathsf{g}(O\setminus Q_t).
\end{eqnarray*}
Via a $C^\infty$ approximation of $f_t$, we see that (\ref{eq5d-})
is valid for $f_t/\sqrt{t}$, and so that
$$
A_\mathsf{g}(O\setminus Q_t)\le C_1A_\mathsf{g}(O)\exp(-c_1t).
$$
This inequality implies
\begin{eqnarray*}
\int_O g_{(O,\mathsf{g})}(z,a)A_\mathsf{g}(z)&=&\int_0^\infty
A_\mathsf{g}(O\setminus Q_t)dt\\
&\le&C_1A_\mathsf{g}(O)\int_0^\infty\exp(-c_1t)dt\\
&=&C_1c_1^{-1}A_\mathsf{g}(O).
\end{eqnarray*}
Thus (ii) holds with $C_2=C_1c_1^{-1}$.

(ii)$\Rightarrow$(iii) Suppose (ii) is valid. To prove (iii), let
$O\in BRD(\mathbb M^2)$ with $C^\infty$ boundary, and
$\lambda_{1,\mathsf{g}}(O)$ be the first nonzero eigenvalue of the
Laplacian operator $\Delta_\mathsf{g}$ for the Dirichlet problem on
$O$. So, if $u\not\equiv 0$ solves
\[
\left\{\begin{array} {r@{\;,\quad}l}
\big(\Delta_\mathsf{g}-\lambda_{1,\mathsf{g}}(O)\big)u(z)=0 & z\in O,\\
u(z)=0 & z\in \partial O,
\end{array}
\right.
\]
then for each $a\in O$ we have
\begin{eqnarray*}
u(a)&=&\int_O g_{(O,\mathsf{g})}(z,a)\Delta_\mathsf{g} u(z)dA_\mathsf{g}(z)\\
&\le&\lambda_{1,\mathsf{g}}(O)\int_O g_{(O,\mathsf{g})}(z,a)u(z)dA_\mathsf{g}(z)\\
&\le&\lambda_{1,\mathsf{g}}(O)\big(\sup_{z\in O}u(z)\big)\int_O g_{(O,\mathsf{g})}(z,a)dA_\mathsf{g}(z)\\
\end{eqnarray*}
whence getting
$$
1\le\lambda_{1,\mathsf{g}}(O)\sup_{a\in O}\int_O
g_{(O,\mathsf{g})}(z,a)dA_\mathsf{g}(z) \le
C_2\lambda_{1,\mathsf{g}}(O)A_\mathsf{g}(O).
$$
Namely, Faber-Krahn's eigenvalue inequality
\begin{equation}\label{eq5e}
\big({\lambda_{1,\mathsf{g}}(O)}\big)^{-1}=\sup\left\{\frac{\int_O
|f|^2dA_{\mathsf{g}}}{\int_O|\nabla_\mathsf{g}
f|^2dA_{\mathsf{g}}}:\quad f\in C^\infty_0(O),\, f\not\equiv
0\right\}\le C_2 A_{\mathsf{g}}(O)
\end{equation}
holds for all $O\in BRD(\mathbb M^2)$. Now, (\ref{eq5e}) and
\cite[Lemma 6.3]{Gri} yield (iii) with
$C_3=2\big(\epsilon(1-\epsilon)C_2\big)^{-1}\le 8(C_2)^{-1}$ where
$\epsilon$ is any given constant in $(0,1)$.

Next, we prove the second part of the conclusion. Note first that if
(iii) holds then according to \cite[Theorem 4.2.6]{Sal} there is a
constant $C_4>0$ such that the heat-kernel-upper-bound inequality
\begin{equation}\label{eq5f1}
H(t,z,a)\le
C_4t^{-1}\exp\Big(-\frac{\big(d_\mathsf{g}(z,a)\big)^2}{8t}\Big)
\end{equation}
holds for all $(z,a,t)\in\mathbb M^2\times\mathbb M^2\times
(0,\infty)$. Here and henceforth, $H(t,z,a)$ stands for the heat
kernel on $\mathbb M^2$ -- that is -- the smallest positive solution
to the heat equation
\[
\left\{\begin{array} {r@{\;,\quad}l} \big(\frac{\partial}{\partial
t}-\Delta_\mathsf{g}\big) H(t,z,a)=0& (t,z,a)\in(0,\infty)\times \mathbb M^2\times\mathbb M^2,\\
H(0,z,a)=\delta_a(z) & (z,a)\in\mathbb M^2\times\mathbb M^2.
\end{array}
\right.
\]
Even more interestingly, this heat kernel indeed describes the
probability of reaching $z$ at time $t$ starting from $a$.
Consequently, when $a\in O$ and $O\in BRD(\mathbb M^2)$ the
integration of $H(t,z,a)$ over $O$ against $dA_\mathsf{g}(z)$ is the
probability $P_a[B_t\in O]$ of the Brownian motion $B_t$ reaching
$O$ at $t$ starting from $a$ on $(\mathbb M^2,\mathsf{g})$, namely,
$$
P_a[B_t\in O]=\int_O H(t,z,a)dA_\mathsf{g}(z).
$$
If
$$
t_O(w)=\inf\{t>0:\ B_t(w)\notin O\}\quad\hbox{and}\quad P_a[t<t_O]
$$
represent the first exit-time at $w$ and the probability that the
Brownian motion begins with $a$ and hits $O$ by $t_O$ respectively,
then the corresponding expectation $E_a[t_O]$ can be formulated as
\begin{equation}\label{eqgr}
\int_O
g_{(O,\mathsf{g})}(z,a)dA_\mathsf{g}(z)=E_a[t_O]=\int_0^\infty
P_a[t<t_O]dt.
\end{equation}

In light of the study done in \cite[Theorem 1.6]{BaOk}, we continue
our proof as follows. The condition (\ref{eq5f1}) and the layer cake
representation (see \cite[p. 26, Theorem 1.13]{LiLo} again) yield
\begin{eqnarray*}
P_a[t<t_O]&\le&\int_{O}H(t,z,a)dA_\mathsf{g}(z)\\
&\le&C_4t^{-1}\int_O\exp\Big(-(8t)^{-1}\big(d_\mathsf{g}(z,a)\big)^2\Big)dA_\mathsf{g}(z)\\
&=&C_4t^{-1}\int_0^\infty A_\mathsf{g}\big(\{z\in O:\
d_\mathsf{g}(z,a)>\tau\}\big)\Big(\frac{d}{d\tau}\exp\big(-(8t)^{-1}\tau^2\big)\Big)d\tau.
\end{eqnarray*}
The foregoing inequality, plus choosing $r_0>0$ such that
$A_\mathsf{g}(O)=A_\mathsf{g}\big(B(a,r_0)\big)$, further gives
\begin{eqnarray*}
P_a[t<t_O]&\le& C_4t^{-1}\int_0^{r_0} A_\mathsf{g}\big(\{z\in
\mathbb M^2:\
d_\mathsf{g}(z,a)>\tau\}\big)\Big(\frac{d}{d\tau}\exp\big(-(8t)^{-1}\tau^2\big)\Big)d\tau\\
&=&C_4t^{-1}\int_{B(a,r_0)}\exp\Big(-(8t)^{-1}\big(d_\mathsf{g}(z,a)\big)^2\Big)dA_\mathsf{g}(z)\\
&=&C_4t^{-1}\int_0^{r_0}\exp\big(-(8t)^{-1}r^2\big)L_\mathsf{g}\big(\partial B(a,r)\big)dr\\
&\le&C_0C_4t^{-1}\int_0^{r_0}\exp(-(8t)^{-1}r^2)rdr\\
&\le& 4C_0C_4\Big(1-\exp\big(-(8t)^{-1}r_0^2\big)\Big).
\end{eqnarray*}
This estimation, along with (\ref{eqgr}) and (\ref{eq5ge}), now
derives
\begin{eqnarray*}
P_a[2t<t_O]&\le&\Big(\sup_{z\in O}P_z[t<t_O]\Big)^2\\
&\le& \left(4C_0C_4\Big(1-\exp\big(-(8
t)^{-1}r_0^2\big)\Big)\right)^2.
\end{eqnarray*}
This immediately produces
\begin{eqnarray*}
\int_O
g_{(O,\mathsf{g})}(z,a)dA_\mathsf{g}(z)&=&A_\mathsf{g}(O)\int_0^\infty
P_a\big[sA_\mathsf{g}(O)<t_O\big]ds\\
&\le&(2 r_0 C_0C_4)^2\int_0^\infty\big(1-\exp(-t^{-1})\big)^2dt\\
&=&\Big(c_0^{-1}(2C_0C_4)^2\int_0^\infty\big(1-\exp(-t^{-1})\big)^2dt\Big)
A_\mathsf{g}(O),
\end{eqnarray*}
namely, (i) holds.
\end{proof}

\begin{remark}\label{r5final} Several more comments on the last theorem are in order:

\item{\rm(i)} In the case of $(\mathbb
M^2,\mathsf{g})=(\mathbb R^2,ds^2)$, the maximal value of $c_1$ in
(\ref{eq5d-}) is $4\pi$ -- this is due to Moser; see also \cite{Mo}.
Moreover, from \cite[Proposition 2]{Ma} and \cite[(2.10)]{Fl} we see
that (\ref{eq5d-}) with $c_1=4\pi$ amounts to
$$
A_{\mathsf{g}}(E)\le A_{\mathsf{g}}(O)\exp\big(-4\pi
\hbox{mod}_O(E)\big)
$$
for any compact $E\subset O$, where
$$
\hbox{mod}_O(E)=\sup\Big\{\Big(\int_O|\nabla_g
f|^2\,dA_g\Big)^{-1}:\ f\in C^\infty_0(O),\ f\ge 1\ \hbox{on}\
E\Big\}.
$$

\item{\rm(ii)} Under the hypotheses of Theorem \ref{t52}, if $K_\mathsf{g}\ge
0$ and $C_3=4(\pi\lambda_{1,\mathcal{N}})^{-1}$ (the Carlen-Loss's
sharp constant in \cite{CaLo}) where $\lambda_{1,\mathcal{N}}$ is
the first non-zero Neumann eigenvalue of $\Delta$ on radial
functions on $\mathbb D$, then $(\mathbb M^2,\mathsf{g})$ is
isometric to $(\mathbb R^2,ds^2)$ -- this result is proved in
\cite[Theorem 1.4]{Xia}. Similarly, if $K_\mathsf{g}\ge 0$ and
$\tau_\mathsf{g}(\mathbb M^2)=1$, then (\ref{eq5c-})/(\ref{eq5c})
holds with the best Euclidean constant, and hence $(\mathbb
M^2,\mathsf{g})$ is isometric to $(\mathbb R^2,ds^2)$ -- see also
\cite[p. 244]{He}. Accordingly, it is our conjecture that this
isometry follows also from the conditions $C_2=(4\pi)^{-1}$ and
$K_\mathsf{g}\ge 0$. Despite being unable to verify this conjecture,
we can obtain a weaker result as follows.

Suppose $K_\mathsf{g}\ge 0$. Then (\ref{eq5d}) yields $H(t,z,a)\le
C_4t^{-1}$ and so by Li-Yau's maximal volume growth theorem in
\cite{LiYau},
$$
\liminf_{r\to\infty}A_\mathsf{g}\big(B(a,r)\big)(\pi r^2)^{-1}\ge
l_0\quad\hbox{for\ \ some\ \ constant}\ \ l_0>0.
$$
A use of Gromov's comparison theorem (cf. \cite[p. 11]{He}) gives
$$
l_0\le A_\mathsf{g}\big(B(a,r)\big)(\pi r^2)^{-1}\le
1\quad\hbox{for\ \ all}\ \ r>0.
$$
Of course, if $l_0=1$ then $(\mathbb M^2,\mathsf{g})$ is isometric
to $(\mathbb R^2,ds^2)$. But, if $l_0<1$ then a result of
Cheeger-Colding in \cite{ChCo} produces that $(\mathbb
M^2,\mathsf{g})$ is diffeomorphic to $(\mathbb R^2,ds^2)$.

\item{\rm(iii)} From \cite[Theorem 3]{Be} and its proof it follows that the
above Nash-Sobolev's inequality holds whenever there exists a
constant $C_5>0$ such that the Log-Sobolev inequality
\begin{equation}\label{eq5g}
\exp\Big(\int_{\mathbb M^2}|f|^2\ln |f|^2dA_\mathsf{g}\Big)\le
C_5\int_{\mathbb M^2}|\nabla_\mathsf{g} f|^2dA_\mathsf{g}
\end{equation}
holds for all $f\in C^\infty_0(\mathbb M^2)$ with $\int_{\mathbb
M^2}|f|^2dA_\mathsf{g}=1$. As well, it is known that (\ref{eq5f})
implies (\ref{eq5g}) -- see \cite{Gri} for example. Moreover, if
$K_\mathsf{g}\ge 0$ and (\ref{eq5g}) holds with $C_5=(e\pi)^{-1}$,
then $(\mathbb M^2,\mathsf{g})$ is isometric to $(\mathbb R^2,ds^2)$
-- see also \cite[Corollary 1.5]{Ni}.

\item{\rm(iv)} When compared with the setting on the flat surface $(\mathbb R^2,ds^2)$,
the requirement (\ref{eq5ge}) is not artificial -- see also
\cite{LiTa} once again. In fact, if $u$ is a bounded $C^\infty$
function on $\mathbb R^2$ then (\ref{eq5ge}) holds on the manifold
$(\mathbb R^2,e^{2u}ds^2)$, and hence the previously-stated five
inequalities: (\ref{eq5d}); (\ref{eq5f}); (\ref{eq5e});
(\ref{eq5f1}); (\ref{eq5g}) are equivalent. Of course, this
equivalence is new even for $u=0$. Besides, the condition
(\ref{eq5ge}) is closely related to the following conclusion.

$(a)$ Li's criterion for the finite total curvature in \cite{Li-}
tells us that if $(\mathbb M^2,\mathsf{g})$ is a complete noncompact
surface with: finite topological type, at most quadratic area growth
-- $\lim_{r\to\infty}r^{-2}A_{\mathsf{g}}(B(a,r))$ existing, and the
Gauss curvature being of one sign at each end, then $\int_{\mathbb
M^2} |K_\mathsf{g}|dA_\mathsf{g}<\infty$.

$(b)$ Conversely, if $(\mathbb M^2,\mathsf{g})$ is a complete
noncompact surface with $\int_{\mathbb M^2}
|K_\mathsf{g}|dA_\mathsf{g}<\infty$ then
\smallskip

$(b_1)$ Hartman's area-length domination in \cite {Ha} induces two
positive constants $c_0^\ast, C_0^\ast$ ensuring
$$
A_\mathsf{g}\big(B(a,r)\big)\le c_0^\ast r^2\quad\hbox{and}\quad
L_\mathsf{g}\big(\partial B(a,r)\big)\le C_0^\ast
r\quad\hbox{for}\quad 0<r<\infty;
$$

$(b_2)$ Shiohama's minimal-area principle in \cite{Sh} gives
\[
\inf_{\mathsf{g}\in\mathcal{M}(\mathbb M^2)}A_\mathsf{g}(\mathbb M^2)=\left\{\begin{array} {r@{\;,\quad}l}4\pi & \chi(\mathbb M^2)=1,\\
-2\pi\chi(\mathbb M^2) & \chi(\mathbb M^2)\le 0,
\end{array}
\right.
\]
where $\mathcal{M}(\mathbb M^2)$ stands for all complete Riemannian
metrics $\mathsf g$ on $\mathbb M^2$ with the next constraint:
\[
\left\{\begin{array} {r@{\;\quad as \quad}l}
K_\mathsf{g}\le 1 & \chi(\mathbb M^2)\ge 0,\\
K_\mathsf{g}\ge -1 & \chi(\mathbb M^2)<0.
\end{array}
\right.
\]
\end{remark}

\bibliographystyle{amsplain}

\end{document}